\documentclass[leqno,a4paper,twoside,11pt]{article}
\usepackage[nointlimits,nosumlimits]{amsmath}
\usepackage{amsfonts,amssymb,amsthm,ifthen}
\usepackage[utf8x]{inputenc}

\newtheorem{Thm}{Theorem}[section]
\newtheorem{Prp}[Thm]{Proposition}
\newtheorem{Lem}[Thm]{Lemma}
\newtheorem{Cor}[Thm]{Corollary}
\newtheorem{Def}[Thm]{Definition}
\newtheorem{Rem}[Thm]{Remark}

\newcommand\id{\mathrm{id}}
\newcommand\dvol{\mathrm{dvol}}
\newcommand\Spin{\mathrm{Spin}}

\newcommand\setsep{\;\big|\;}
\newcommand\dbar{\overline{\partial}}

\newcommand\laplace{\triangle}

\renewcommand\vec[1]{\boldsymbol #1}
\newcommand\vvarphi{\vec\varphi}
\newcommand\vpsi{\vec\psi}
\newcommand\vx{\vec x}
\newcommand\vy{\vec y}
\newcommand\vz{\vec z}

\newcommand\mA{\mathcal A}
\newcommand\mC{\mathcal C}
\newcommand\mD{\mathcal D}
\newcommand\mJ{\mathcal J}
\newcommand\mM{\mathcal M}
\newcommand\mU{\mathcal U}
\newcommand\bC{\mathbb C}
\newcommand\bN{\mathbb N}
\newcommand\bR{\mathbb R}
\newcommand\bZ{\mathbb Z}
\newcommand\ox{\overline{x}}
\newcommand\tvarphi{\tilde{\varphi}}
\newcommand\tpsi{\tilde{\psi}}
\newcommand\tD{\tilde{D}}
\newcommand\tJ{\tilde{J}}
\newcommand\tz{\tilde{z}}

\newcommand{\dd}[2][]{\frac{\partial #1}{\partial #2}}
\newcommand{\abs}[1]{\left| #1 \right|}
\newcommand{\norm}[1]{\left|\left| #1 \right|\right|}

\newcommand{\scal}[3][]{\ifthenelse{\equal{#1}{}}{
  \left\langle #2,\,#3 \right\rangle
}{\ifthenelse{\equal{#1}{(}}{
  \left( #2,\,#3 \right)
}{\ifthenelse{\equal{#1}{[}}{
  \left[ #2,\,#3 \right]
}{
  #1\left( #2,\,#3 \right)
}}}}

\renewcommand{\title}[1]{\vbox{\center\LARGE{\textsc{#1}}}\vspace{5mm}}
\renewcommand{\author}[1]{\vbox{\center\large{\textsc{#1}}}\vspace{5mm}}
\newcommand{\address}[1]{\vbox{\center\em#1}}
\newcommand{\email}[1]{\vbox{\center\tt#1}\vspace{5mm}}

\begin{document}

\title{Compactness for Holomorphic Supercurves}

\author{Josua Groeger$^1$}

\address{Humboldt-Universit\"at zu Berlin, Institut f\"ur Mathematik,\\
  Rudower Chaussee 25, 12489 Berlin, Germany }

\email{$^1$groegerj@mathematik.hu-berlin.de}


\begin{abstract}
\noindent
We study the compactness problem for moduli spaces of holomorphic supercurves
which, being motivated by supergeometry, are perturbed such as to allow for
transversality.
We give an explicit construction of limiting objects for sequences
of holomorphic supercurves and prove that, in important cases, every such sequence
has a convergent subsequence provided that a
suitable extension of the classical energy is uniformly bounded.
This is a version of Gromov compactness. Finally, we introduce a topology on the
moduli spaces enlarged by the limiting objects which makes these spaces
compact and metrisable.
\end{abstract}

removal of singularities, Gromov topology.

\section{Introduction}

Holomorphic supercurves were introduced in \cite{Gro11a} as a natural generalisation
of holomorphic curves to supergeometry. Holomorphic curves, in turn, have been a
powerful tool in the study of symplectic manifolds since the seminal work of Gromov (\cite{Gro85}).
In general, solution sets of nonlinear elliptic differential equations often lead
to interesting algebraic invariants. In the case of holomorphic curves, these are known
as Gromov-Witten invariants.
Here, one typically encounters two problems: First, the moduli spaces in question are
wanted to have a nice structure (which is a transversality problem) and, second,
they are usually not compact but need to be compactified.
Examples for the occurrence and solution of both issues include the aforementioned
Gromov-Witten invariants (cf. \cite{MS04}), invariants of Hamiltonian group actions
as introduced in \cite{CGMS02} and symplectic field theory (cf. \cite{EGH00} and \cite{Dra04}).
In the case of holomorphic supercurves, the transversality problem was solved in \cite{Gro11b}
by making the defining equations depend on a connection $A$ such that the corresponding linearised operator
is generically surjective. This perturbed definition is referred to as an $(A,J)$-holomorphic supercurve.

To fix notation, let us briefly recall the relevant background.
Let $\Sigma$ be a connected and closed Riemann surface with a fixed complex structure $j$.
Moreover, for simplicity, we also fix a Riemann metric $h$ in the conformal class
corresponding to $j$.
By a standard result, biholomorphic diffeomorphisms of $\Sigma$ agree with its
conformal automorphisms (cf. Sec. A.3 of \cite{BP92}).
For our purposes, the most important example is the sphere $\Sigma=S^2$ which we identify,
via stereographic projection, with two copies of $\bC$ glued along $\bC^*$ via the transition map
$z\mapsto\frac{1}{z}$. Under this identification, the standard metric on $S^2$ is a constant multiple
of the \emph{Fubini-Study metric} $h_{(s,t)}:=\frac{1}{(1+s^2+t^2)^2}(ds^2+dt^2)$
on $\bC\cup\{\infty\}$, and conformal automorphisms, known as \emph{Möbius transformations},
have the form $z\mapsto\frac{az+b}{cz+d}$ with $ad-bc=1$ and $a,b,c,d\in\bC$.
As our target space, we let $(X,\omega)$ denote a compact symplectic manifold
of dimension $2n$ and fix an $\omega$-tame (or $\omega$-compatible) almost complex structure $J$.
Every such structure $J$ determines a Riemann metric $g_J$.
A \emph{($J$-)holomorphic curve} is a smooth map $\varphi:\Sigma\rightarrow X$
such that $\dbar_J\varphi:=\frac{1}{2}(d\varphi+J\circ d\varphi\circ j)=0$ holds.

\begin{Def}[\cite{Gro11b}]
\label{defAJHolomorphicSupercurve}
Let $L\rightarrow\Sigma$ be a holomorphic line bundle and $A$ be a connection on $X$.
An \emph{$(A,J)$-holomorphic supercurve} is a pair $(\varphi,\psi)$, consisting of a smooth map
$\varphi\in C^{\infty}(\Sigma,X)$ and a smooth section $\psi\in\Gamma(\Sigma,L\otimes_J\varphi^*TX)$,
for brevity denoted $(\varphi,\psi):(\Sigma,L)\rightarrow X$, such that
\begin{align*}
\dbar_J\varphi=0\;,\qquad\mD^{A,J}_{\varphi}\psi
:=\left(\nabla^{A,J}\psi_{\theta}\right)^{0,1}\cdot\theta+\psi_{\theta}\cdot(\dbar\theta)=0
\end{align*}
holds where, for $U\subseteq\Sigma$ sufficiently small, we fix a nonvanishing section
$\theta\in\Gamma(U,L)$ and let $\psi_{j\theta}\in\Gamma(U,\varphi^*TX)$ be
such that $\psi_j=\theta\cdot\psi_{j\theta}$ holds. Here, $\dbar$ denotes the
usual Dolbeault operator on $L$.
\end{Def}

Based on Def. \ref{defAJHolomorphicSupercurve}, this article solves the compactness problem in important cases
and may be read independent of \cite{Gro11a} and \cite{Gro11b}.
It is organised as follows.
In Sec. \ref{secProperties}, we show that the defining equations are conformally
invariant and introduce the super energy as a conformally invariant extension of the classical energy.
In Sec. \ref{secRemovalOfSingularities}, we prove removability of isolated singularities for holomorphic supercurves with
finite super energy by means of mean value type inequalities and an isoperimetric inequality for local holomorphic supercurves.
In Sec. \ref{secBubbling}, we study how the bubbling off of classical holomorphic spheres affects
a sequence of holomorphic supercurves. We prove, in particular, that the rescaling can be chosen such
that there is no loss of super energy.
In Sec. \ref{secGromovCompactness}, we introduce stable supercurves and prove that
every sequence of holomorphic superspheres with uniformly bounded super energy
has a subsequence that converges to such an object. This is Gromov compactness.
Finally, we define a compact and metrisable topology on the moduli spaces
that describes this convergence.
Put together, our transversality and compactness results raise hope to be able to construct
new invariants or at least to find new expressions for existing ones in subsequent work.

\subsection{Gromov Compactness for Holomorphic Curves}

We recall some basic facts about compactness in the case of holomorphic curves.
Consult \cite{MS04} as well as the references therein for details.
Since the holomorphicity condition $\dbar_J\varphi=0$ is conformally invariant,  
there is a canonical action of the group $G$ of conformal automorphisms of $\Sigma$ on the moduli space
of simple $J$-holomorphic curves representing a given homology class $\beta\in H_2(X,\bZ)$.
Since $G$ is typically not compact, this space is, in general, non-compact either.
Another failure of compactness lies in the formation of bubbles. Let us first recall
the notion of \emph{energy} for smooth maps $\varphi:\Sigma\rightarrow X$:
\begin{align}
\label{eqnHarmonicAction}
E(\varphi):=\frac{1}{2}\int_{\Sigma}\abs{d\varphi}_{h,g_J}^2\dvol_{\Sigma}
\end{align}
On the one hand, it coincides with the \emph{harmonic action functional} $\mA(\varphi)=E(\varphi)$
and, by the energy identity, is a topological invariant.
On the other hand, it is important for the bubbling analysis to be sketched next.

Consider $\Sigma=S^2$ and let $\varphi^{\nu}$ be a sequence of holomorphic curves such that the energy
$E(\varphi^{\nu})$ is uniformly bounded.
Then a suitable subsequence converges to a holomorphic sphere $\varphi$ on $S^2$
with a finite number of points $z_j\in S^2$ removed and, at each $z_j$, another holomorphic
sphere (a \emph{bubble}) is attached in the limit. The proof is based on a conformal rescaling argument,
using removability of isolated singularities for holomorphic curves and conformal invariance
of the energy $E(\varphi,U)$ for $U\subseteq S^2$.
The limiting objects are obtained by an inductive argument over the bubbling process,
which terminates by the following result: There is a constant $\hbar>0$ such that
$E(\varphi)\geq\hbar$
for every nonconstant $J$-holomorphic sphere $\varphi:S^2\rightarrow X$. One can thus show that every
sequence of holomorphic curves has a subsequence that converges to a \emph{stable map}
(in the sense of Kontsevich)
\begin{align}
\label{eqnStableMap}
(\vvarphi,\vz)=(\{\varphi_{\alpha}\}_{\alpha\in T},\{z_{\alpha\beta}\in S^2\}_{\alpha E\beta},
\{\alpha_i,z_i\in S^2\}_{1\leq i\leq n})
\end{align}
that is a collection of holomorphic curves modelled over a labelled tree $(T,E,\Lambda)$
with compatibility conditions for the edges. This result is called \emph{Gromov compactness}.
Here, the \emph{marked points} $z_i$ are not needed in the first place but added
for obtaining an evaluation map, from which the Gromov-Witten invariants are then built.

There is a natural notion of \emph{equivalence} of two stable maps
$(\vvarphi,\vz)\sim(\vvarphi',\vz')$ by a tree isomorphism and a collection
of Möbius transformations. We define the \emph{moduli spaces}
\begin{align}
\mM_{0,n}(X;J)\quad,\qquad\mM_{0,n}(X,\beta;J)\quad,\qquad\overline{\mM}_{0,n}(X,\beta;J)
\end{align}
of stable maps with $n$ marked points, those which represent $\beta$,
and the corresponding space of equivalence classes, respectively.
By a more general Gromov compactness theorem, every sequence in $\mM_{0,n}(X;J)$ with
uniformly bounded energy has a convergent subsequence, and the same is true for every
sequence in $\mM_{0,n}(X,\beta;J)$. Moreover, limits are unique up to equivalence.
Gromov convergence can be measured by a ''distance'' function
$\scal[\rho_{\varepsilon}]{(\vvarphi,\vz)}{(\vvarphi',\vz')}$
for $(\vvarphi,\vz),(\vvarphi',\vz')\in\mM_{0,n}(X,\beta;J)$, where $\varepsilon>0$
is a sufficiently small constant. By means of $\rho_{\varepsilon}$, one can define
a topology on $\overline{\mM}_{0,n}(X,\beta;J)$, called \emph{Gromov topology}, with
respect to which these spaces are compact and metrisable, and
convergence is equivalent to Gromov convergence.

\section{Conformal Invariance and the Super Energy}
\label{secProperties}

In this section, we show conformal invariance of $(A,J)$-holomorphic supercurves
and introduce a conformally invariant extension of the energy (\ref{eqnHarmonicAction}).

For this purpose, let us first study an important class of holomorphic line bundles
on Riemann surfaces.
Let $\mathrm{Spin}(2)\cong S^1$ denote the $2$-dimensional spin group
and, in the following, fix a spin structure $\Spin(\Sigma)$ on $\Sigma$.
We define, for $d\in\bZ^*$, representations
\begin{align*}
\mu_d:\Spin(2)\rightarrow GL(\bC)\cong\bC^*\;&,\qquad\mu_d(s)(v)=s^d\cdot v
\end{align*}
on $\bC$ and consider the complex line bundles $L_d:=\Spin(\Sigma)\times_{(\Spin(2),\mu_d)}\bC$
of degree $\deg L_d=d(1-g)$ where $g$ denotes the genus of $\Sigma$.
On the sphere $S^2$, which carries a unique spin structure,
the (isomorphism classes of) holomorphic line bundles are exactly the bundles $L_d$.
Of particular importance are the (half-)spinor bundles
$S^+=L_{-1}$ and $S^-=L_1$ as well as the (co)tangent bundles
$T^*\Sigma=L_{-2}$ and $T\Sigma=L_2$. In general, we may identify
$L_{d}\cong (L_1)^d\cong (L_{-1})^{-d}\cong(S^+)^{-d}$,
and the spin structure thus induces a holomorphic structure on each $L_d$.
By a simple calculation, the Hermitian metric $H$ on $\bC$, defined by
$\scal[H]{x}{y}:=\Re(\ox y)$, is invariant under the action
$\mu_d$ of $\Spin(2)$ for every $d\in\bZ^*$. We thus yield a global
metric as follows.

\begin{Lem}
\label{lemSpinInvariantFormsGlobal}
Let $s\in\Gamma(U\subseteq\Sigma,\Spin(\Sigma))$ be a local section such that
elements of $(L_d)_z$ may be written $[s_z,v]$, $[s_z,w]$ with $v,w\in\bC$.
Then, prescribing
\begin{align*}
\scal[H]{[s_z,v]}{[s_z,w]}:=\scal[[]{s_z}{\scal[H]{v}{w}}
\end{align*}
yields a well-defined Hermitian bundle metric on $L_d$.
\end{Lem}

In general, let $L$ be a complex line bundle and $d\in\bZ^*$.
We let $l_d$ denote the canonical bundle map
$l_d:L\rightarrow L^d$, which is constructed as follows. Note first that any bundle atlas of $L$
induces a bundle atlas of $L^d$: Considering $L$ as a collection $\{U_{\alpha}\times\bC\}$
for open subsets $U_{\alpha}\subseteq\Sigma$ with cocycles
$g_{\alpha\beta}:U_{\alpha}\cap U_{\beta}\rightarrow\bC^*$,
$L^d$ may be identified with the same collection $\{U_{\alpha}\times\bC\}$ with cocycles $g_{\alpha\beta}^d$.
Now identify $v\in L_{|U_{\alpha}}$ with a complex number and prescribe $l_d(v):=v^d$.
This prescription is well-defined.

\begin{Lem}
\label{lemBundleHomomorphismLift}
Let $m,\tilde{m}:\Sigma\rightarrow\Sigma$ be holomorphic maps which are homotopic to the identity
$\id:\Sigma\rightarrow\Sigma$. Then, for every $d\in\bZ^*$, there exists a homomorphism
$\sqrt{dm}^d:L_d\rightarrow L_d$ of holomorphic line bundles, which is unique upon prescribing
$\sqrt{\id}^d:=\id$. Moreover, $\sqrt{d(m\circ\tilde{m})}^d=\sqrt{dm}^d\circ\sqrt{d\tilde{m}}^d$
holds. If $m$, $\tilde{m}$ are bijective, then so is $\sqrt{dm}^d$ and
$(\sqrt{dm}^d)^{-1}=\sqrt{d(m^{-1})}^d$ and
$(\sqrt{d(m\circ\tilde{m})}^d)^{-1}=(\sqrt{d\tilde{m}}^d)^{-1}\circ(\sqrt{dm}^d)^{-1}$.
In this case, we may view $\sqrt{dm}^d$ as a bundle isomorphism
$\sqrt{dm}^d:L_d\stackrel{\cong}{\rightarrow}m^*L_d$.
\end{Lem}

\begin{proof}
The differential $dm:(L_1)^2\rightarrow(L_1)^2$ is homotopic to the identity since $m$ is homotopic to the
identity and, moreover, complex linear since $m$ is holomorphic.
Therefore, the existence of $\sqrt{dm}^d$ as a homomorphism of complex line bundles with the properties stated follows
from standard arguments involving covering spaces (cf. Sec. 1.3 in \cite{Hat01}).
$\sqrt{dm}^d$ is holomorphic because $dm$ is.
\end{proof}

\begin{Lem}
\label{lemLiftConformal}
Let $m:\Sigma\rightarrow\Sigma$ be a holomorphic map
which is homotopic to the identity.
Then the induced bundle homomorphisms $\sqrt{dm}^d:L_d\rightarrow L_d$ are conformal
with respect to $H$. More precisely,
\begin{align*}
\scal[H_{m(z)}]{\sqrt{dm}^d(v)}{\sqrt{dm}^d(w)}=(\lambda_m(z))^{\frac{d}{2}}\cdot\scal[H_z]{v}{w}
\end{align*}
holds for $v,w\in (L_d)_z$, where $\lambda_m$ is the conformal factor from $m^*h=\lambda_m\cdot h$.
\end{Lem}

\begin{proof}
Let $s$ be a local section of $\Spin(\Sigma)$ in a neighbourhood of $z\in U\subseteq\Sigma$.
Then, by definition, $L_d$ is trivial on $U$ upon identifying $[s,v]\in\Spin(\Sigma)\times_{\mu_d}\bC=L_d$
with the complex number $v$. The same holds for a neighbourhood of $m(z)$, and analogous for the tangent bundle
$T\Sigma\cong\Spin(\Sigma)\times_{\mu_2}\bR^2$.
With respect to these trivialisations, the maps $dm$ and $\sqrt{dm}^d$ correspond to multiplication with complex
numbers, denoted (unambiguously) by the same symbols, such that $\abs{dm}=\sqrt{\lambda_m}$.
\end{proof}

Every Möbius transformation of the sphere is homotopic to the identity. In this case,
the hypotheses of Lem. \ref{lemBundleHomomorphismLift} and Lem. \ref{lemLiftConformal}
are, therefore, satisfied.
Now let $\varphi:\Sigma\rightarrow X$ be a smooth map and consider the (twisted) complex vector bundles
$L_d\otimes_J\varphi^*TX\cong L_d\otimes_{\bC}\varphi^*T^{1,0}X$.
Let $m:\Sigma\rightarrow\Sigma$ be a conformal automorphism which is homotopic to the identity.
We denote the lift from Lem. \ref{lemBundleHomomorphismLift} by $\Phi_m:=\sqrt{dm}^d:L_d\rightarrow L_d$.
Its inverse induces identifications (bundle isomorphisms)
$\Phi_m^{-1}:m^*L_d\stackrel{\cong}{\rightarrow} L_d$ and
\begin{align*}
\Phi_m^{-1}:m^*L_d\otimes_J(\varphi\circ m)^*TX\stackrel{\cong}{\rightarrow}L_d\otimes_J(\varphi\circ m)^*TX
\end{align*}
denoted by the same symbol.

\begin{Def}
\label{defTransformedCurve}
Consider a function $\varphi\in C^{\infty}(\Sigma,X)$ and a section $\psi\in\Gamma(L_d\otimes_J\varphi^*TX)$.
We define the transformed objects under $m$ by
\begin{align*}
\tvarphi:=\varphi\circ m\;,\qquad
\tpsi:=\Phi_m^{-1}\circ\psi\circ m\in\Gamma(L_d\otimes_J\tvarphi^*TX)
\end{align*}
\end{Def}

If $\tilde{m}$ is another conformal automorphism homotopic to the identity, Lem. \ref{lemBundleHomomorphismLift}
implies that the concatenation satisfies
\begin{align}
\label{eqnTransformedConcatenation}
\varphi\circ(m\circ\tilde{m})=(\varphi\circ m)\circ\tilde{m}\;,\quad
\Phi_{m\circ\tilde{m}}^{-1}\circ\psi\circ(m\circ\tilde{m})
=\Phi_{\tilde{m}}^{-1}\circ\left(\Phi_m^{-1}\circ\psi\circ m\right)\circ\tilde{m}
\end{align}

\begin{Prp}[Conformal Invariance]
\label{prpSuperConformalInvariance}
Let $m:\Sigma\rightarrow\Sigma$ be a conformal automorphism which is homotopic to the identity and
$(\varphi,\psi)$ be a pair, consisting of a function $\varphi\in C^{\infty}(\Sigma,X)$ and a section
$\psi\in\Gamma(\Sigma,L_d\otimes_J\varphi^*TX)$. Then
\begin{align*}
\mD^{A,J}_{\tvarphi}(\tpsi)=\Phi_m^{-1}\circ\mD_{\varphi}^{A,J}\psi\circ dm
\in\Omega^{0,1}(L_d\otimes_J\tvarphi^*TX)
\end{align*}
holds. In particular, the defining equations for an $(A,J)$-holomorphic supercurve
$(\varphi,\psi):(\Sigma,L_d)\rightarrow X$ are conformally invariant:
\begin{align*}
\dbar_J\varphi=0\iff\dbar_J\tvarphi=0\;,\qquad
\mD_{\varphi}^{A,J}\psi=0\iff\mD_{\tvarphi}^{A,J}(\tpsi)=0
\end{align*}
\end{Prp}

\begin{proof}
By assumption, $dm$ commutes with $\dbar$ and, by Lem. \ref{lemBundleHomomorphismLift},
the same holds for $\Phi_m^{-1}$, such that
$\dbar\circ\Phi_m^{-1}(\theta\circ m)=\Phi_m^{-1}\circ(\dbar\theta)\circ dm$
follows. Let $z\in\Sigma$. In general, the pullback connection satisfies the identity
\begin{align*}
\left(\nabla^A(\xi\circ m)\right)_z=\left(\nabla^A\xi\right)_{m(z)}\circ d_zm
\end{align*}
for every vector field $\xi\in\Gamma(\varphi^*TX)$ and the composition
$\xi\circ m\in(\varphi\circ m)^*TX$. Therefore
\begin{align*}
\left(\nabla^{A,J}(\psi_{\theta}\circ m)\right)^{0,1}
=\left(\nabla^{A,J}\psi_{\theta}\right)^{0,1}\circ dm\in\Omega^{0,1}((\varphi\circ m)^*TX)
\end{align*}
With this preparation, the statement follows from a straightforward calculation.
\end{proof}

We will next introduce a conformally invariant extension of the energy
(\ref{eqnHarmonicAction}) which involves the sectional part $\psi$.
Let $g=g_J$ denote the Riemann metric on $X$ induced by $\omega$ and $J$,
and let $H$ denote the Hermitian metric on $L_d$ from Lem. \ref{lemSpinInvariantFormsGlobal}.
$H$ and $g$ together induce a bundle metric,
denoted $\scal[(]{}{}$, on $L_d\otimes_{\bR}\varphi^*TX\cong L_d\otimes_{\bC}\varphi^*T^{\bC}X$,
which descends to the subbundle $L_d\otimes_{\bC}\varphi^*T^{1,0}X$.
It is Hermitian since $H$ and $g$ are Hermitian (i.e. $g$ is $J$-orthogonal).
A short calculation yields
\begin{align}
\label{eqnH2}
\scal[(]{\psi}{\psi'}=\frac{1}{2}\scal[g]{\psi_{\theta}}{\psi'_{\theta}}\cdot\scal[H]{\theta}{\theta}
+\frac{1}{2}\scal[g]{\psi_{\theta}}{J\psi'_{\theta}}\cdot\scal[H]{\theta}{i\theta}
\end{align}
where $\theta$ and $\psi_{\theta}$ are local sections of $L_d$ and $\varphi^*TX$,
respectively, and we write sections of $L_d\otimes_J\varphi^*TX$ in the (local) form
$\psi=\frac{1}{2}(\psi_{\theta}\theta-iJ\psi_{\theta}\theta)$.
The following lemma follows immediately from Lem.
\ref{lemLiftConformal} and $m^*\dvol_{\Sigma}=\lambda_m\cdot\dvol_{\Sigma}$.

\begin{Lem}
\label{lemSuperEnergyInvariant}
In the situation of Def. \ref{defTransformedCurve}, let $d\neq 0$
and $\psi,\psi'\in\Gamma(L_d\otimes_J\varphi^*TX)$. Then
\begin{align*}
\int_U\scal[(]{\psi}{\psi'}^{-\frac{2}{d}}\dvol_{\Sigma}
=\int_{m^{-1}(U)}\scal[(]{\tpsi}{\tpsi'}^{-\frac{2}{d}}\dvol_{\Sigma}
\end{align*}
holds for $U\subseteq\Sigma$, whenever both sides are defined.
In other words, the integral is conformally invariant.
\end{Lem}

\begin{Rem}
Similarly, one shows that in the case $d=0$, i.e. $L=L_0=\underline{\bC}$,
the following integral is conformally invariant:
\begin{align*}
\int_U\scal[(]{\nabla\psi}{\nabla\psi'}_{\varphi^*g,h}\dvol_{\Sigma}
=\int_{m^{-1}(U)}\scal[(]{\nabla\tpsi}{\nabla\tpsi'}_{\tvarphi^*g,h}\dvol_{\Sigma}
\end{align*}
where $\nabla$ is any connection on $TX$ (e.g. the Levi-Civita connection of $g$), and we denote
the induced connection on $\varphi^*TX=\underline{\bC}\otimes_J\varphi^*TX$ by the same symbol.
\end{Rem}

\begin{Def}
Let $U\subseteq\Sigma$ be an open subset. For a pair $(\varphi,\psi)$, consisting of a function
$\varphi\in C^{\infty}(\Sigma,X)$ and a section $\psi\in\Gamma(L_d\otimes_J\varphi^*TX)$ such that $d\neq 0$,
we define the \emph{energies}
\begin{align*}
&E(\varphi,U):=\frac{1}{2}\int_U\abs{d\varphi}^2\,\dvol_{\Sigma}\;,\qquad
E(\psi,U):=\frac{1}{2}\int_U\abs{\psi}^{-\frac{4}{d}}\,\dvol_{\Sigma}\\
&E(\varphi,\psi,U):=E(\varphi,U)+E(\psi,U)
\end{align*}
where the norm $\abs{\psi}^2=\scal[(]{\psi}{\psi}$ is defined by the bundle metric (\ref{eqnH2}).
We call $E(\varphi,\psi):=E(\varphi,\psi,S^2)$ the \emph{super energy} of $(\varphi,\psi)$.
\end{Def}

The super energy of a pair $(\varphi,\psi)$ extends the energy of $\varphi$ as defined in
(\ref{eqnHarmonicAction}). Note that we do not treat $E(\varphi,\psi)$ as an action functional, however.
By Lem. \ref{lemSuperEnergyInvariant}, it is conformally invariant:
\begin{align}
\label{eqnSuperEnergyInvariance}
E(\varphi,U)=E(\varphi\circ m,m^{-1}(U))\;,\qquad
E(\psi,U)=E(\psi\circ m,m^{-1}(U))
\end{align}
where we abbreviate $\psi\circ m:=\tpsi=\Phi_m^{-1}\circ\psi\circ m$, using the notation
from Def. \ref{defTransformedCurve}. By (\ref{eqnTransformedConcatenation}), it is then clear that
$\psi\circ (m\circ\tilde{m})=(\psi\circ m)\circ\tilde{m}$ holds for the concatenation of two
conformal automorphisms $m,\tilde{m}$ homotopic to the identity.
We summarise the most important cases $d=-2$ and $d=-1$:
\begin{align*}
E_{d=-1}(\psi,U)=\frac{1}{2}\int_U\abs{\psi}^4\,\dvol_{\Sigma}\;,\qquad
E_{d=-2}(\psi,U)=\frac{1}{2}\int_U\abs{\psi}^2\,\dvol_{\Sigma}
\end{align*}

\section{Removal of Singularities}
\label{secRemovalOfSingularities}

In this section, we prove the following theorem about the removability of
isolated singularities.

\begin{Thm}[Removal of Singularities]
\label{thmRemovalOfSingularities}
Let $U\subseteq\Sigma$ be an open subset, $p\in U$ and $(\varphi,\psi):(U\setminus\{p\},L_d)\rightarrow X$
be an $(A,J)$-holomorphic supercurve on $U$ with finite super energy $E(\varphi,\psi,U)<\infty$.
Then, in the cases $d=-2$ and $d=-1$, $(\varphi,\psi)$ extends smoothly over $p$.
\end{Thm}

The corresponding statement for the underlying holomorphic curve $\varphi$ is a classical result,
stated e.g. as Thm. 4.1.2 in \cite{MS04}. In the proof for the section $\psi$,
we use mean value type inequalities and a local isoperimetric
inequality.
Thm. \ref{thmRemovalOfSingularities} is entirely local. Recall from
\cite{Gro11b} that every $(A,J)$-holomorphic supercurve
is a \emph{local holomorphic supercurve} upon choosing conformal coordinates on $\Sigma$
and trivialising the bundle $L$ by a local holomorphic section.
In the following, we denote by $z=s+it$ the standard coordinates on $\bC\cong\bR^2$.

\begin{Def}
\label{defLocalHolomorphicSupercurve}
Let $U\subseteq\bC$ be an open set and $p>2$. Then a pair of functions $(\varphi,\psi)\in W^{1,p}(U,\bR^{2n})$
is called a \emph{local holomorphic supercurve} if there are functions
$J\in C^{l+1}(V,\bR^{2n\times 2n})$ and $D\in C^l(V\times\bR^{2n},\bR^{2n\times 2n})$,
where $V\subseteq\bR^{2n}$ is an open set with $\varphi(U)\subseteq V$, $J^2=-\id$ holds
and $D$ is linear in the second component, such that
\begin{align*}
\partial_s\varphi+J(\varphi)\cdot\partial_t\varphi=0\;,\qquad
\partial_s\psi+J(\varphi)\cdot\partial_t\psi+(D(\varphi)\cdot\partial_s\varphi)\cdot\psi=0
\end{align*}
or, equivalently, upon abbreviating $\tilde{J}:=J\circ\varphi$ and $\tilde{D}:=(D(\varphi)\cdot\partial_s\varphi)$
\begin{align}
\label{eqnLocalHolomorphicCurve}
\partial_s\varphi(z)+\tilde{J}(z)\cdot\partial_t\varphi(z)&=0\\
\label{eqnLocalHolomorphicSupercurve}
\partial_s\psi(z)+\tilde{J}(z)\cdot\partial_t\psi(z)+\tilde{D}(z)\cdot\psi(z)&=0
\end{align}
for every $z\in U$.
\end{Def}

This is a good place for collecting some estimates for later use.
Consult \cite{Wer02} or \cite{Dob06} for a standard treatment.
We denote the Sobolev norms for maps $f:U\rightarrow V$ with $U\subseteq\bR^n$
and $V\subseteq\bR^m$ by $\norm{f}_{U,p}:=\norm{f}_{L^p(U,V)}$ and
$\norm{f}_{U,k,p}:=\norm{f}_{W^{k,p}(U,V)}$.
This includes the case $p=\infty$. If the reference to $U$ is clear, we also abbreviate the norms
towards $\norm{f}_p=\norm{f}_{0,p}$ and $\norm{f}_{k,p}$, respectively.
Hölder's inequality reads $\norm{uv}_1\leq\norm{u}_p\norm{v}_q$ for
$1/p+1/q=1$. Applied to $p=\frac{3}{2}$ and $q=3$, we thus yield
\begin{align}
\label{eqnHoelder43}
\norm{fg}_{\frac{4}{3}}=\left(\norm{\abs{f}^{\frac{4}{3}}\abs{g}^{\frac{4}{3}}}_1\right)^{\frac{3}{4}}
\leq\left(\norm{\abs{f}^{\frac{4}{3}}}_{\frac{3}{2}}\norm{\abs{g}^{\frac{4}{3}}}_3\right)^{\frac{3}{4}}
=\norm{f}_2\norm{g}_4
\end{align}
Different Lebesgue norms may be estimated against each other:
Let $1\leq p\leq q\leq\infty$ and assume that $U$ is bounded. Then there is a constant $C>0$,
depending on $U$, $p$ and $q$ such that
\begin{align}
\label{eqnLebesgue}
\norm{f}_{U,p}\leq C\cdot\norm{f}_{U,q}
\end{align}
By definition of the Sobolev norms, an analogous estimate holds for $\norm{\cdot}_{U,k,p}$ and
$\norm{\cdot}_{U,k,q}$.
Next, the Sobolev embedding theorem states that, for $1\leq kp<n$ and provided that $U$ is a bounded
and has a Lipschitz boundary, there is a constant $C>0$,
depending on $U$, $k$ and $p$ such that
\begin{align}
\label{eqnSobolev}
\norm{f}_{U,np/(n-kp)}\leq C\cdot\norm{f}_{U,k,p}
\end{align}
For our purposes, $U$ is usually a subset of $\bR^2$, such that $n=2$.
Estimate (\ref{eqnSobolev}), applied to $k=1$ and $p=4/3$, then reads $\norm{f}_4\leq C\norm{f}_{1,\frac{4}{3}}$.
Sobolev spaces embed into spaces of differentiable functions as follows.
Let $k>0$ and $p>n$, and assume that $U$ is bounded and has a Lipschitz boundary.
Then there is a constant $C>0$, depending on $U$, $k$ and $p$ such that
\begin{align}
\label{eqnMorrey}
\norm{f}_{C^{k-1}(U)}\leq C\cdot\norm{f}_{U,k,p}\;,\qquad
\sup_{\substack{x,y\in U\\x\neq y}}\frac{\abs{f(x)-f(y)}}{\abs{x-y}^{\mu}}\leq C\cdot\norm{df}_{U,p}
\end{align}
for $\mu:=1-n/p$. The second estimate is known as Morrey's inequality,
and the inclusion $W^{k,p}(U)\subseteq C^{k-1}(U)$ is compact.

\subsection{The Local Super Energy}

We translate boundedness of the super energy into local form next. For the cases $d=-2$ and $d=-1$,
the sectional part $E(\psi)$ involves $L^2$ and $L^4$-norms, respectively, which depend on
the global bundle metrics from (\ref{eqnH2}) and may be estimated by the
corresponding norms for functions $\bR^2\rightarrow\bR^{2n}$.

For later use, we need a sharper estimate, treating different $\varphi$ and $J$ simultaneously, as follows.
For $z\in S^2$ and $p\in X$, we consider the scalar product $\scal[(]{\cdot}{\cdot}$ on the vector space
$L_z\otimes_{J_p}T_pX$, which is constructed from $H_z$ and $g^J_p$ analogous to the bundle metric (\ref{eqnH2}).
Fixing a nonvanishing local section $\theta\in\Gamma(U,L)$ and coordinates $(V\subseteq X,x^1,\ldots,x^{2n})$ on $X$,
there is a canonical identification
\begin{align}
\label{eqnUniformIdentification}
L_z\otimes_{J_p}T_pX\rightarrow\bR^{2n}\;,
\qquad\frac{1}{2}\left(\dd{x^j}|_p-iJ_p\dd{x^j}|_p\right)\cdot\theta_z\mapsto e_j
\end{align}
for $z\in U$ and $p\in V$, where $e_j$ denotes the $j$-th standard basis vector of $\bR^{2n}$.
Moreover, choosing $\theta$ such that $\scal[H]{\theta}{\theta}\equiv 1$ and $\scal[H]{\theta}{j\theta}\equiv 0$
(which is possible for every proper open subset $U\subseteq S^2$), we obtain
\begin{align}
\label{eqnUniformMetric}
\scal[(]{e_j}{e_k}_{L_z\otimes_{J_p}T_pX}=\frac{1}{2}\scal[g^J_{p}]{\dd{x^j}}{\dd{x^k}}
\end{align}
Writing $v=v^j\cdot e_j$, a standard estimate yields the following lemma.

\begin{Lem}
\label{lemLocalGlobalEstimate}
Let $K\subseteq X$ be a compact subset which lies completely within a coordinate chart of $X$.
Let $U\subseteq S^2$ be a proper open subset and $\theta\in\Gamma(U,L)$ be a section
with respect to which (\ref{eqnUniformIdentification}) and (\ref{eqnUniformMetric}) hold.
Then there are constants $m,M>0$, depending (continuously) on $J$, $\theta$ and the
$X$-coordinates chosen, such that for all $z\in U$, $p\in K$ and $v\in L_z\otimes_{J_p}T_pX\cong\bR^{2n}$
\begin{align*}
m\abs{v}_{L_z\otimes_{J_p}T_pX}\leq\abs{v}_{\bR^{2n}}\leq M\abs{v}_{L_z\otimes_{J_p}T_pX}
\end{align*}
holds, where $\abs{\cdot}_{\bR^{2n}}$ denotes the standard norm on $\bR^{2n}$.
\end{Lem}

For the rest of this section, we shall denote by $\scal{\cdot}{\cdot}:=\scal{\cdot}{\cdot}_{\bR^{2n}}$
the (constant) standard scalar product on $\bR^{2n}$ and by $\abs{\cdot}:=\abs{\cdot}_{\bR^{2n}}$
the induced norm, unless otherwise stated.
We introduce a local version of the super energy next, which is defined such that,
on the one hand, it is bounded above by the ordinary super energy in the cases $d=-2,-1$
(by Lem. \ref{lemSuperEnergyLocal} below) and, on the other hand,
its boundedness suffices for the proof of Thm. \ref{thmRemovalOfSingularities}, which will
become clear in the course of this section.

\begin{Def}
Let $U\subseteq\bR^2$ be an open subset and $\varphi,\psi\in C^1(U,\bR^{2n})$. Then the \emph{local super energy}
of $(\varphi,\psi)$ on $U$ is the following integral.
\begin{align*}
E^{\mathrm{loc}}(\varphi,\psi,U):=\int_U\left(\abs{d\varphi}^2+\abs{\psi}^2+\abs{d\psi}^2\right)ds\,dt
\end{align*}
\end{Def}

Here and in the following, $s+it$ denotes the standard coordinates on $\bC\cong\bR^2$.
Moreover, for brevity we shall often omit writing ''$ds\,dt$'' in integrals over a subset of $\bR^2$.
The next lemma is borrowed from App. B in \cite{MS04}. It is proved by means of an inequality due to
Calderon and Zygmund \cite{CZ52}, \cite{CZ56}.

\begin{Lem}[Elliptic Bootstrapping, weak form]
\label{lemEllipticBootstrappingWeak}
Let $U'\subseteq U\subseteq\bC$ be open sets such that $\overline{U'}\subseteq U$ and $1<p<\infty$.
Then, for every constant $c_0>0$, there is a constant $c>0$ such that the following holds.
Assume $J\in W^{1,\infty}(U,\bR^{2n\times 2n})$ satisfies $J^2=-\id$ and $\norm{J}_{U,1,\infty}\leq c_0$.
Then, for every $u\in W^{1,p}_{\mathrm{loc}}(U,\bR^{2n})$, the following estimate holds.
\begin{align*}
\norm{u}_{U',1,p}\leq c\left(\norm{\partial_su+J\partial_tu}_{U,0,p}+\norm{u}_{U,0,p}\right)
\end{align*}
\end{Lem}

In the following lemmas, we use the notations $J$, $D$, $\tJ$ and $\tD$ as in Def.
\ref{defLocalHolomorphicSupercurve}.

\begin{Lem}
\label{lemL2NormEstimateReverse}
Let $U'\subseteq U\subseteq\bR^2$ be bounded open sets such that $\overline{U'}\subseteq U$ and
$(\varphi,\psi)\in C^1(U,\bR^{2n})$ be a local holomorphic supercurve of regularity class $C^1$.
Then there is a constant $C\geq 0$, depending on $\norm{\tJ}_{U,1,\infty}$ and $\norm{\tD}_{U,0,\infty}$
(but not on $\psi$) such that
\begin{align*}
\int_{U'}\abs{d\psi}^2\leq C\cdot\int_U\abs{\psi}^2
\end{align*}
\end{Lem}

\begin{proof}
By (\ref{eqnLocalHolomorphicSupercurve}), the hypotheses of Lem. \ref{lemEllipticBootstrappingWeak}
are satisfied for $p=2$ and $J$ replaced by $\tJ$, and the statement follows from the
resulting estimate as follows.
\begin{align*}
\left(\int_{U'}\abs{d\psi}^2\right)^{\frac{1}{2}}
\leq c\left(\norm{\partial_s\psi+\tJ\partial_t\psi}_{U,2}+\norm{\psi}_{U,2}\right)
\leq c\left(\norm{\tD}_{U,\infty}+1\right)\cdot\left(\int_U\abs{\psi}^2\right)^{\frac{1}{2}}
\end{align*}
\end{proof}

\begin{Lem}
\label{lemSuperEnergyLocal}
Let $(\varphi,\psi):(S^2,L_d)\rightarrow X$ be an $(A,J)$-holomorphic supercurve.
Let $U'\subseteq U\subseteq\bR^2$ be bounded open sets such that $\overline{U'}\subseteq U$
and such that $(\varphi,\psi)$ may be considered as a local holomorphic
supercurve on $U$. Then, in the cases $d=-2$ and $d=-1$, there is a constant
$C\geq 0$, depending on $\norm{\tJ}_{U,1,\infty}$ and $\norm{\tD}_{U,0,\infty}$
(but not on $\psi$), such that
\begin{align*}
E^{\mathrm{loc}}(\varphi,\psi,U')\leq C\cdot\left(E(\varphi,\psi,U)+E(\varphi,\psi,U)^{\frac{1}{2}}\right)
\end{align*}
In particular, the local super energy is bounded if the super energy is.
\end{Lem}

\begin{proof}
By Lem. \ref{lemLocalGlobalEstimate}, we may estimate the standard norm on $\bR^{2n}$
by the bundle norm. In the case $d=-2$, we thus obtain a constant $c>0$,
depending only on the geometry, such that the following estimate holds.
\begin{align*}
\int_U\abs{\psi}^2 ds\,dt\leq c\cdot\int_U\abs{\psi}_{L_d\otimes_J\varphi^*TX}^2\,\dvol_{S^2}
=2c\cdot E(\varphi,\psi,U)
\end{align*}
In the case $d=-1$, the analogous argument and, moreover, the Lebesgue estimate
(\ref{eqnLebesgue}) with $p=2$ and $q=4$ yield constants $c_1,c_2>0$, depending
only on the geometry, such that the following estimate holds.
\begin{align*}
\int_U\abs{\psi}^2\leq c_1\cdot\left(\int_U\abs{\psi}^4\right)^{\frac{1}{2}}
\leq c_2\cdot\left(\int_U\abs{\psi}_{L_d\otimes_J\varphi^*TX}^4\,\dvol_{S^2}\right)^{\frac{1}{2}}
=\sqrt{2}\,c_2\cdot E(\varphi,\psi,U)^{\frac{1}{2}}
\end{align*}

In either case, the integral over $\abs{d\varphi}^2$ is, similarly, bounded above by a multiple
of $E(\varphi,U)$. Finally, the remaining estimate for the integral of the term $\abs{d\psi}^2$
(over a smaller set $U'$ and with a constant as stated in the hypotheses) follows directly from
Lem. \ref{lemL2NormEstimateReverse}. This proves the statement.
\end{proof}

\subsection{Mean Value Inequalities}

For the proof of Thm. \ref{thmRemovalOfSingularities}, and also for the bubbling analysis in the next
section, we need the mean value inequalities for local holomorphic supercurves to be developed next.
They both follow from a lemma due to Heinz. Let $\laplace:=\partial_s^2+\partial_t^2$ denote the standard
Laplacian on $\bR^2$ and $B_r:=B_r(0)$ be the (open) disc around $0$ with radius $r$.

\begin{Lem}[The Heinz Trick, Lem. A.1 in \cite{RS01}]
\label{lemHeinzTrick}
Let $r>0$ and $a,b\geq 0$. Let $w:B_r\rightarrow\bR$ be a $C^2$-function that satisfies
the inequalities
\begin{align*}
\laplace w\geq -a-bw^2\;,\qquad w\geq 0\;,\qquad\int_{B_r}w<\frac{\pi}{16b}
\end{align*}
Then the estimate $w(0)\leq\frac{ar^2}{8}+\frac{8}{\pi r^2}\int_{B_r}w$ holds.
\end{Lem}

\begin{Lem}[First Mean Value Inequality]
\label{lemMVI1}
Let $0\in U\subseteq\bR^2$ be a bounded open set, $(\varphi,\psi)\in C^{l+1}(U,\bR^{2n})$ be
a local holomorphic supercurve and $r>0$ small enough such that $B_r:=B_r(0)\subseteq U$. Then
\begin{align*}
\int_{B_r}\abs{\psi}^2<\frac{\pi}{16}\implies
\abs{\psi(0)}^2\leq\frac{ar^2}{8}+\frac{8}{\pi r^2}\int_{B_r}\abs{\psi}^2
\end{align*}
where $a\geq 0$ is a constant depending on $\norm{\tD}_{U,1,\infty}$ and
$\norm{\tJ}_{U,1,\infty}$ (but not on $\psi$).
\end{Lem}

\begin{proof}
With $w:=\abs{\psi}^2$, we have
\begin{align}
\label{eqnLaplaceMean}
\laplace w=2\scal{\laplace\psi}{\psi}+2\abs{\partial_s\psi}^2+2\abs{\partial_t\psi}^2
\end{align}

For the rest of the proof, we shall write $D$ and $J$ instead of $\tD$ and $\tJ$, respectively.
Then, the holomorphicity condition (\ref{eqnLocalHolomorphicSupercurve}) implies
\begin{align*}
\laplace\psi&=\partial_s(-J\cdot\partial_t\psi-D\cdot\psi)+\partial_t(J\partial_s\psi+JD\psi)\\
&=\left(-\partial_sD+\partial_t(JD)\right)\psi+\left(-D+\partial_tJ\right)\partial_s\psi
+\left(-\partial_sJ+JD\right)\partial_t\psi\\
&=:L\psi+M\partial_s\psi+N\partial_t\psi
\end{align*}
with $L,M,N:U\rightarrow\bR^{2n\times 2n}$ depending only on $J$ and $D$ and their first derivatives.
The previous calculation implies the estimate
\begin{align*}
\scal{\laplace\psi}{\psi}
&\leq\left(\abs{L}\abs{\psi}+\abs{M}\abs{\partial_s\psi}+\abs{N}\abs{\partial_t\psi}\right)\abs{\psi}\\
&=\abs{L}\abs{\psi}^2+2\left(\frac{1}{2}\abs{M}\abs{\psi}\right)\abs{\partial_s\psi}
+2\left(\frac{1}{2}\abs{N}\abs{\psi}\right)\abs{\partial_t\psi}\\
&\leq\abs{L}\abs{\psi}^2+\frac{1}{4}\abs{M}^2\abs{\psi}^2+\abs{\partial_s\psi}^2
+\frac{1}{4}\abs{N}^2\abs{\psi}^2+\abs{\partial_t\psi}^2
\end{align*}
and, using (\ref{eqnLaplaceMean}), we obtain
\begin{align*}
\laplace w&\geq-2\abs{\scal{\laplace\psi}{\psi}}+2\abs{\partial_s\psi}^2+2\abs{\partial_t\psi}^2\\
&\geq-\left(2\abs{L}+\frac{1}{2}\abs{M}^2+\frac{1}{2}\abs{N}^2\right)\abs{\psi}^2\\
&\geq-\left(2\norm{L}_{U,\infty}+\frac{1}{2}\norm{M}_{U,\infty}^2
+\frac{1}{2}\norm{N}_{U,\infty}^2\right)\abs{\psi}^2\\
&=:-2\sqrt{a}\cdot w\geq-a-w^2
\end{align*}
By the last estimate, the hypotheses of Heinz' Lemma \ref{lemHeinzTrick} are satisfied with $b:=1$,
thus providing the estimate which was to show.
\end{proof}

\begin{Lem}[Second Mean Value Inequality]
\label{lemMVI2}
Let $0\in U\subseteq\bR^2$ be a bounded open set, $(\varphi,\psi)\in C^{l+1}(U,\bR^{2n})$ be
a local holomorphic supercurve with $l\geq 2$ and $r>0$ small enough such that
$B_{2r}:=B_{2r}(0)\subseteq U$. Then
\begin{align*}
\int_{B_{2r}}\abs{\psi}^2<\frac{\pi}{16}\,,\;\int_{B_r}\abs{d\psi}^2<\frac{\pi}{64}\implies
\abs{d\psi(0)}^2\leq\frac{1}{4}+cr^2+dr^4+\frac{8}{\pi r^2}\int_{B_r}\abs{d\psi}^2
\end{align*}
where $c,d\geq 0$ depend on $\norm{\tD}_{U,2,\infty}$ and $\norm{\tJ}_{U,2,\infty}$ (but not on $\psi$).
\end{Lem}

\begin{proof}
Let $w:=\abs{d\psi}^2$ and set $\xi:=\partial_s\psi$ and $\eta:=\partial_t\psi$.
Then $w=\abs{\xi}^2+\abs{\eta}^2$ and
\begin{align}
\label{eqnLaplaceMean2}
\laplace w=2\scal{\laplace\xi}{\xi}+2\scal{\laplace\eta}{\eta}
+2\abs{\partial_s\xi}^2+2\abs{\partial_s\eta}^2+2\abs{\partial_t\xi}^2+2\abs{\partial_t\eta}^2
\end{align}
For the terms involving $\laplace\xi$ and $\laplace\eta$ note that
$\partial_t\partial_s\psi=\partial_s\partial_t\psi$
and thus $\partial_t\xi=\partial_s\eta$ holds, whence
$\partial_s\partial_t\eta=\partial_t\partial_s\eta=\partial_t\partial_t\xi$. Therefore,
with $L$,$M$ and $N$ as in the proof of the previous lemma, we yield
\begin{align*}
\laplace\xi&
=\partial_s(\partial_s\xi+\partial_t\eta)+\partial_t\partial_t\xi-\partial_s\partial_t\eta\\
&=\partial_s(\partial_s\partial_s\psi+\partial_t\partial_t\psi)\\
&=\partial_s\left(L\psi+M\partial_s\psi+N\partial_t\psi\right)\\
&=(\partial_sL)\psi+\left(L+\partial_sM\right)\xi+(\partial_sN)\eta+M(\partial_s\xi)+N(\partial_s\eta)\\
&=:X\psi+Y\xi+Z\eta+M(\partial_s\xi)+N(\partial_s\eta)
\end{align*}
The previous calculation implies the following estimate.
\begin{align*}
\scal{\laplace\xi}{\xi}
&\leq\abs{X\psi+Y\xi+Z\eta+M(\partial_s\xi)+N(\partial_s\eta)}\abs{\xi}\\
&\leq\frac{1}{4}\left(\norm{X}_{U,\infty}^2+4\norm{Y}_{U,\infty}+\norm{Z}_{U,\infty}^2
+\norm{M}_{U,\infty}^2+\norm{N}_{U,\infty}^2+4\right)
\left(\abs{\xi}^2+\abs{\eta}^2\right)\\
&\qquad+\abs{\psi}^2+\abs{\partial_s\xi}^2+\abs{\partial_s\eta}^2\\
&=:2\sqrt{P}\cdot w+\abs{\psi}^2+\abs{\partial_s\xi}^2+\abs{\partial_s\eta}^2\\
&\leq w^2+(P+\abs{\psi}^2)+\abs{\partial_s\xi}^2+\abs{\partial_s\eta}^2
\end{align*}
where $P$ only depends on the infinity norms over $U$ of derivatives up to second order of $\tJ$ and $\tD$.
By analogous estimates, one can further show that there is a constant $Q$ with corresponding dependencies
such that
\begin{align*}
\scal{\laplace\eta}{\eta}\leq w^2+(Q+\abs{\psi}^2)+\abs{\partial_t\xi}^2+\abs{\partial_t\eta}^2
\end{align*}
Thus, using (\ref{eqnLaplaceMean2}), we obtain
\begin{align*}
\laplace w&\geq -2\abs{\scal{\laplace\xi}{\xi}}-2\abs{\scal{\laplace\eta}{\eta}}
+2\abs{\partial_s\xi}^2+2\abs{\partial_s\eta}^2+2\abs{\partial_t\xi}^2+2\abs{\partial_t\eta}^2\\
&\geq -4w^2-(2P+2Q+4\abs{\psi}^2)
\end{align*}
By hypothesis we have, for $\abs{\rho}< r$,
\begin{align*}
\int_{B_r(\rho e^{i\theta})}\abs{\psi}^2\leq\int_{B_{2r}(0)}\abs{\psi}^2<\frac{\pi}{16}
\end{align*}
and thus, applying Lem. \ref{lemMVI1} with $0$ replaced by $\rho e^{i\theta}\in B_r$,
we obtain
\begin{align*}
\abs{\psi(\rho e^{i\theta})}^2\leq\frac{ar^2}{8}+\frac{8}{\pi r^2}\int_{B_r(\rho e^{i\theta})}\abs{\psi}^2
\leq\frac{ar^2}{8}+\frac{1}{2r^2}
\end{align*}
Therefore
\begin{align*}
\laplace w\geq -4w^2-\left(2P+2Q+\frac{ar^2}{2}+\frac{2}{r^2}\right)=:-4w^2-A
\end{align*}
follows. By Heinz' Lemma \ref{lemHeinzTrick}, applied with $A$ and $b:=4$, we thus yield
\begin{align*}
\abs{d\psi(0)}^2&\leq\frac{Ar^2}{8}+\frac{8}{\pi r^2}\int_{B_r}\abs{d\psi}^2\\
&=\frac{1}{4}+\frac{P+Q}{4}r^2+\frac{a}{16}r^4+\frac{8}{\pi r^2}\int_{B_r}\abs{d\psi}^2
\end{align*}
which was to show.
\end{proof}

\begin{Cor}
\label{corMVI}
Let $0\in U\subseteq\bR^2$ be a bounded open set,
$(\varphi,\psi)\in C^{l+1}(U\setminus\{0\},\bR^{2n})$ be a local holomorphic supercurve
where $l\geq 2$ and such that $E^{\mathrm{loc}}(\varphi,\psi,U)<\infty$.
Then there is a constant $r_0>0$, depending on $\psi$, and a constant $C> 0$, depending
on $\norm{\tD}_{U,2,\infty}$ and $\norm{\tJ}_{U,2,\infty}$ and $r_0$ (but not on $\psi$),
such that $B_{3r_0}\subseteq U$ and for all $r<r_0$
\begin{align*}
\abs{d\psi(r e^{i\theta})}^2\leq C+\frac{8}{\pi r^2}\int_{B_{2r}}\abs{d\psi}^2
\end{align*}
\end{Cor}

\begin{proof}
By the assumption on the boundedness of the local super energy, there is $0<r_0<1$ (sufficiently small)
such that $B_{3r_0}\subseteq U$ and
\begin{align*}
\int_{B_{3r_0}}\abs{\psi}^2<\frac{\pi}{16}\;,\qquad\int_{B_{2r_0}}\abs{d\psi}^2<\frac{\pi}{64}
\end{align*}
holds. By the inclusions $B_{2r}(r e^{i\theta})\subseteq B_{3r_0}$ and $B_{r}(r e^{i\theta})\subseteq B_{2r_0}$
for $r<r_0$, the hypotheses of Lem. \ref{lemMVI2} are, therefore, satisfied with $0$ replaced by $r e^{i\theta}$,
and the following inequalities result.
\begin{align*}
\abs{d\psi(r e^{i\theta})}^2
\leq\frac{1}{4}+cr^2+dr^4+\frac{8}{\pi r^2}\int_{B_{r}(r e^{i\theta})}\abs{d\psi}^2
\leq\frac{1}{4}+cr_0^2+dr_0^4+\frac{8}{\pi r^2}\int_{B_{2r}}\abs{d\psi}^2
\end{align*}
\end{proof}

\subsection{Proof of the Singularity Theorem}

Based on the preparations until now, we now come to the actual proof of
Thm. \ref{thmRemovalOfSingularities}. Its core is enshrined in
Prp. \ref{prpRemovalOfSingularities} below which, in turn, is based on the
mean value inequalities from the previous subsection and a local isoperimetric inequality
to be established. We begin with another local estimate.

\begin{Lem}
\label{lemDpsiEstimate}
Let $U'\subseteq U\subseteq\bR^2$ be bounded open sets with a smooth boundary
such that $\overline{U'}\subseteq U$,
and let $(\varphi,\psi)\in C^2(U,\bR^{2n})$ be a local holomorphic supercurve of
regularity class $C^2$. Then there is a constant $C\geq 0$, depending on
$\norm{D(\varphi)}_{U,1,\infty}$ and $\norm{\tJ}_{U,1,\infty}$ (but not on $\psi$) such that
the following estimate holds.
\begin{align*}
\int_{U'}\abs{\tD\psi}^2\leq C\int_{U'}\left(\abs{\psi}^2+\abs{d\psi}^2\right)\cdot\int_U\abs{d\varphi}^2
\end{align*}
\end{Lem}

\begin{proof}
By Hölder's inequality, we get the following inequality.
\begin{align*}
\int_{U'}\abs{\tD\psi}^2\leq\norm{\tD}^2_{U',4}\norm{\psi}^2_{U',4}
\end{align*}
Further, using the Lebesgue and Sobolev embeddings (\ref{eqnLebesgue}) and
(\ref{eqnSobolev}), we yield constants $c_1,c_2>0$, depending only on the geometry, such that
\begin{align*}
\int_{U'}\abs{\tD\psi}^2&\leq c_1\norm{\tD}^2_{U',1,\frac{4}{3}}\norm{\psi}^2_{U',1,\frac{4}{3}}\\
&\leq c_2\norm{\tD}^2_{U',1,2}\norm{\psi}^2_{U',1,2}\\
&=c_2\int_{U'}\left(\abs{\tD}^2+\abs{d\tD}^2\right)\cdot\int_{U'}\left(\abs{\psi}^2+\abs{d\psi}^2\right)
\end{align*}
We estimate
\begin{align*}
\abs{\tD}^2&=\abs{D(\varphi)\cdot\partial_s\varphi}^2
\leq\norm{D(\varphi)}_{U,\infty}^2\abs{\partial_s\varphi}^2
\leq\norm{D(\varphi)}_{U,\infty}^2\abs{d\varphi}^2
=:c_3\cdot\abs{d\varphi}^2
\end{align*}
and
\begin{align*}
\abs{\partial_s\tD}^2
\leq\norm{D(\varphi)}^2_{U,1,\infty}\left(\abs{d\varphi}^2+\abs{d^2\varphi}^2\right)
\end{align*}
and analogous for $\abs{\partial_t\tD}^2$, providing a constant $c_4>0$,
depending on $\norm{D(\varphi)}_{U,1,\infty}$, such that
\begin{align*}
\abs{d\tD}^2\leq c_4\left(\abs{d\varphi}^2+\abs{d^2\varphi}^2\right)
\end{align*}

By the estimates for $\abs{\tD}^2$ and $\abs{d\tD}^2$ thus obtained, it remains to show that
the integral over $\abs{d^2\varphi}^2$ is bounded above by (a constant multiple of) $\int_U\abs{d\varphi}^2$.
We set $v:=\partial_s\varphi$ and differentiate the (local) definition (\ref{eqnLocalHolomorphicCurve})
for holomorphic curves, to obtain
\begin{align*}
\partial_s v+\tJ\partial_t v+(\partial_s\tJ)\tJ v=0
\end{align*}
By this equation, the hypotheses of Lem. \ref{lemEllipticBootstrappingWeak} are satisfied
with $J$ and $u$ replaced by $\tJ$ and $v$, respectively, and $p=2$. By the resulting
inequality, we yield a constant $c_5>0$ depending on $\norm{\tJ}_{U,1,\infty}$ such that
the following estimate holds.
\begin{align*}
\int_{U'}\abs{\partial_s\partial_s\varphi}^2&\leq\norm{v}_{U',1,2}^2
\leq c_5\left(\norm{\partial_sv+\tJ\partial_tv}_{U,2}+\norm{v}_{U,2}\right)^2\\
&\leq c_5\left(\norm{(\partial_s\tJ)\tJ}_{U,\infty}+1\right)^2\int_U\abs{d\varphi}^2
=:c_6\cdot\int_U\abs{d\varphi}^2
\end{align*}
Moreover, we obtain analogous estimates for $\partial_s\partial_s\varphi$ replaced with
$\partial_s\partial_t\varphi=\partial_t\partial_s\varphi$ and $\partial_t\partial_t\varphi$, respectively.
This shows $\int_{U'}\abs{d^2\varphi}^2\leq 4c_6\int_U\abs{d\varphi}^2$
and thus finishes the proof.
\end{proof}

We prove a (local) isoperimetric inequality for holomorphic supercurves next.
In order to state the result, we need the following yoga.

\begin{Rem}
\label{remSuperIsoperimetricInequality}
In general, let $J_0\in\bR^{2n\times 2n}$ be a matrix such that $J_0^2=-\id$. We may assume
without loss of generality
that the (constant) scalar product $\scal{\cdot}{\cdot}$ on $\bR^{2n}$ is Hermitian
(i.e. satisfies $\scal{J_0 v}{J_0 w}=\scal{v}{w}$ for all $v,w\in\bR^{2n}$)
for, if not, replace $\scal{\cdot}{\cdot}$ by the average
$\scal{v}{w}':=\frac{1}{2}\left(\scal{v}{w}+\scal{J_0v}{J_0w}\right)$.
On the other hand, any $J_0$-Hermitian scalar product $\scal{\cdot}{\cdot}$ induces a
symplectic form $\omega_0$ on $\bR^{2n}$ by prescribing
$\scal[\omega_0]{v}{w}:=\scal{v}{J_0 w}$
such that $J_0$ is $\omega_0$-compatible.
\end{Rem}

\begin{Lem}[Local Isoperimetric Inequality]
\label{lemSuperIsoperimetricInequality}
Let $0\in U\subseteq\bR^2$ be a bounded open set, $l\geq 2$, and
$(\varphi,\psi)\in C^{l+1}(U\setminus\{0\},\bR^{2n})$ be a local holomorphic supercurve
such that $E^{\mathrm{loc}}(\varphi,\psi,U)<\infty$ and $\varphi$ extends over $0$
to a map $\varphi\in C^{l+1}(U,\bR^{2n})$.
Let $\omega_0$ denote the symplectic form on $\bR^{2n}$ induced by
$J_0:=\tJ(0)$ as in Rem. \ref{remSuperIsoperimetricInequality}.
Then there is a constant $r_0>0$, depending on $\psi$, and a constant $c>0$, depending
on $r_0$, such that $B_{3r_0}\subseteq U$ and, for all $r<r_0$,
\begin{align*}
\int_{B_r}\psi^*\omega_0\leq c\cdot l(\gamma_r)^2
:=c\cdot\left(\int_0^{2\pi}\abs{\dot{\gamma_r}(\theta)}d\theta\right)^2
\end{align*}
holds, where $\gamma_r:S^1\rightarrow X$ denotes the loop $\gamma_r(\theta):=\psi(r e^{i\theta})$.
\end{Lem}

\begin{proof}
By assumption, the hypotheses of Cor. \ref{corMVI} are satisfied, which yields constants $r_0>0$
and $C>0$ (the latter depending on $\varphi$) such that, for $r<r_0$
\begin{align*}
\abs{d\psi(re^{i\theta})}^2\leq C+\frac{8}{\pi r^2}\int_{B_{2r}}\abs{d\psi}^2
\end{align*}
holds and, therefore, we obtain the estimates
\begin{align*}
\abs{\dot{\gamma_r}(\theta)}^2=\frac{r^2}{2}\abs{d\psi(re^{i\theta})}^2
\leq Cr^2+\frac{8}{\pi}\int_{B_{2r}}\abs{d\psi}^2
\end{align*}
and
\begin{align*}
l(\gamma_r)^2=\left(\int_0^{2\pi}d\theta\abs{\dot{\gamma_r}(\theta)}\right)^2
\leq\int_0^{2\pi}d\theta\abs{\dot{\gamma_r}(\theta)}^2
\leq 2\pi Cr^2+16\int_{B_{2r}}\abs{d\psi}^2
\end{align*}
Since the local super energy as well as the constant $C$ are, by hypothesis, bounded,
the last inequality implies that $l(\gamma_r)$ goes to zero as $r\rightarrow 0$.

Now let $\psi_{\rho}:=\psi_{\gamma_{\rho}}:B_{\rho}\rightarrow X$ denote any continuous extension
of $\gamma_{\rho}$. Then Stokes' theorem implies that, for every $0<\rho<r< r_0$, the following identity holds.
\begin{align}
\label{eqnStokesIdentity}
\int_{B_r\setminus B_{\rho}}\psi^*\omega_0+\int_{B_{\rho}}\psi_{\rho}^*\omega_0=\int_{B_r}\psi_r^*\omega_0
\end{align}
By the classical isoperimetric inequality (e.g. stated as Lem. 4.4.3 in \cite{MS04}), there is
a constant $c>0$ such that
$\abs{\int_{B_{\rho}}\psi_{\rho}^*\omega_0}\leq c\cdot l(\gamma_{\rho})^2$ holds,
which goes to zero as $\rho\rightarrow 0$. Therefore, taking the limit $\rho\rightarrow 0$
in (\ref{eqnStokesIdentity}), we obtain
\begin{align*}
\int_{B_r}\psi^*\omega_0=\int_{B_r}\psi_r^*\omega_0\leq c\cdot l(\gamma_r)^2
\end{align*}
using the isoperimetric inequality again.
\end{proof}

\begin{Prp}
\label{prpRemovalOfSingularities}
Let $0\in U\subseteq\bR^2$ be a bounded open set, $l\geq 2$, and
$(\varphi,\psi)\in C^{l+1}(U\setminus\{0\},\bR^{2n})$ be a local holomorphic supercurve
such that $E^{\mathrm{loc}}(\varphi,\psi,U)<\infty$ and $\varphi$ extends over $0$
to a map $\varphi\in C^{l+1}(U,\bR^{2n})$.
Then there are constants $r_0>0$, $0<B<\infty$ and $0<\nu<1$, depending on $(\varphi,\psi)$ and
the geometry, such that $B_{3r_0}\subseteq U$ and, for all $r<r_0$ and $\theta\in[0,2\pi)$,
the following estimate holds.
\begin{align*}
\abs{d\psi(r,\theta)}^2\leq\frac{B}{r^{2-2\nu}}
\end{align*}
\end{Prp}

\begin{proof}
Let $\omega_0$ denote the symplectic form on $\bR^{2n}$ induced by
$J_0:=\tJ(0)$ as in Rem. \ref{remSuperIsoperimetricInequality}. Then
\begin{align*}
\abs{d\psi}^2&=\abs{\partial_s\psi+J_0\partial_t\psi}^2-2\scal{\partial_s\psi}{J_0\partial_t\psi}\\
&=\abs{\tD\psi+(\tJ-J_0)\partial_t\psi}^2-2(\psi^*\omega_0)(\partial_s,\partial_t)\\
&\leq 2\abs{\tD\psi}^2+2\abs{\tJ-J_0}^2\abs{d\psi}^2-2(\psi^*\omega_0)(\partial_s,\partial_t)
\end{align*}
To make further estimates, let $r_0>0$ be the constant of Lem. \ref{lemSuperIsoperimetricInequality},
whose hypotheses are, by assumption, satisfied, and let $r<r_0$. By Taylor's theorem,
there is $\xi\in B_{r_0}\subseteq U$ such that $\tJ(r e^{i\theta})-J_0=d_{\xi}\tJ[r e^{i\theta}]$ and, therefore,
\begin{align*}
\abs{\tJ-J_0}^2(r e^{i\theta})=\abs{d_{\xi}\tJ[r e^{i\theta}]}^2
&\leq\abs{d_{\xi}\tJ}^2r^2\leq\norm{\tJ}_{U,1,\infty}r^2
\end{align*}
We define
\begin{align*}
\varepsilon^{\varphi}(r):=\int_{B_r}\abs{d\varphi}^2\;,\qquad
\varepsilon^{\psi}(r):=\int_{B_r}\abs{d\psi}^2=\int_0^rd\rho\,\rho\int_0^{2\pi}d\theta\abs{d\psi}^2
\end{align*}
and $\varepsilon(r):=\varepsilon^{\varphi}(2r)+\varepsilon^{\psi}(r)$. Then,
using the estimates for $\abs{d\psi}^2$ and $\abs{\tJ-J_0}^2$ just established and
Lem. \ref{lemDpsiEstimate} applied with $U'=B_r$ and $U=B_{2r}$, we yield
a constant $C>0$, depending on $\varphi$ and the geometry but not on $\psi$, such that
the following estimate holds.
\begin{align*}
\varepsilon(r)&=\varepsilon^{\varphi}(2r)+\int_{B_r}\abs{d\psi}^2\\
&\leq\varepsilon^{\varphi}(2r)+2\int_{B_r}\abs{\tD\psi}^2+2\int_{B_r}\abs{J-J_0}^2\abs{d\psi}^2
-2\int_{B_r}(\psi^*\omega_0)(\partial_s,\partial_t)\\
&\leq\varepsilon^{\varphi}(2r)+2C\int_{B_r}(\abs{\psi}^2+\abs{d\psi}^2)\cdot\int_{B_{2r}}\abs{d\varphi}^2
+2Cr^2\int_{B_r}\abs{d\psi}^2-2\int_{B_r}(\psi^*\omega_0)(\partial_s,\partial_t)
\end{align*}
Since, by assumption, the local super energy is bounded, there is a further constant $C_2>0$ such that
\begin{align*}
\varepsilon(r)&\leq C_2\cdot\varepsilon^{\varphi}(2r)+C_2r^2\cdot\varepsilon^{\psi}(r)
-2\int_{B_r}(\psi^*\omega_0)(\partial_s,\partial_t)
\end{align*}
By Lem. \ref{lemSuperIsoperimetricInequality}, the following estimate holds with a constant $c>0$
depending on $r_0$.
\begin{align*}
\int_{B_r}\psi^*\omega_0&\leq c\cdot l(\gamma_r)^2
=\frac{cr^2}{2}\left(\int_0^{2\pi}\abs{d\psi(r,\theta)}d\theta\right)^2\\
&\leq cr^2\int_0^{2\pi}d\theta\abs{d\psi(r,\theta)}^2
=cr\cdot\dot{\varepsilon}^{\psi}(r)
\end{align*}
A similar consideration yields a constant $c_2>0$ such that
$\varepsilon^{\varphi}(r)\leq c_2r\cdot\dot{\varepsilon}^{\varphi}(r)$ (cf. the proof
of Thm. 4.1.2 in \cite{MS04} for details).
Since $\dot{\varepsilon}^{\varphi}\geq 0$ and $\dot{\varepsilon}^{\psi}\geq 0$, we thus obtain
\begin{align*}
\varepsilon(r)\leq C_2\cdot\varepsilon^{\varphi}(2r)+C_2r^2\cdot\varepsilon^{\psi}(r)+2cr\cdot\dot{\varepsilon}^{\psi}(r)
\leq 2c_2 C_2r\cdot\dot{\varepsilon}(r)+C_2r^2\cdot\varepsilon(r)+2cr\cdot\dot{\varepsilon}(r)
\end{align*}
Shrinking $r_0$ if necessary, we may assume that $C_2r^2<\frac{1}{2}$. Then
\begin{align*}
\varepsilon(r)&\leq 2(2c_2 C_2+2c)r\cdot\dot{\varepsilon}(r)=:2br\cdot\dot{\varepsilon}(r)
\end{align*}
follows for $r$ sufficiently small. Therefore, we obtain
\begin{align*}
\frac{\dot{\varepsilon}(r)}{\varepsilon(r)}\geq\frac{1}{2br}=\frac{2\nu}{r}
\end{align*}
with $\nu:=\frac{1}{4b}$ which satisfies, without loss of generality, $\nu<1$
(for, otherwise, replace $b$ by a larger number).
Integrating the inequality from $r$ to $r_1$, this implies
\begin{align*}
\left(\frac{r_1}{r}\right)^{2\nu}\leq\frac{\varepsilon(r_1)}{\varepsilon(r)}
\end{align*}
and hence, setting $c_3:=r_1^{-2\nu}\varepsilon(r_1)$, we yield
$\varepsilon^{\psi}(r)\leq\varepsilon(r)\leq c_3r^{2\nu}$.
For $r$ sufficiently small, Cor. \ref{corMVI} then provides a constant $C_3>0$ such that
\begin{align*}
\abs{d\psi(r,\theta)}^2\leq C_3+\frac{8}{\pi r^2}\varepsilon^{\psi}(2r)
\leq C_3+\frac{A}{r^{2-2\nu}}\leq\frac{B}{r^{2-2\nu}}
\end{align*}
for suitable constants $A,B>0$ which are independent of $r$ and $\theta$.
\end{proof}

\begin{Lem}
\label{lemHoelderContinuous}
Let $0\in U\subseteq\bR^2$ be an open set, $v\in C^1(U\setminus\{0\},\bR^{2n})$ and
$r_0>0$, $0<B<\infty$, $0<\nu<1$ be constants such that $B_{r_0}\subseteq U$ and, for all $0<r<r_0$
and $\theta\in[0,2\pi)$, the following estimate holds.
\begin{align*}
\abs{dv(r,\theta)}^2\leq\frac{B}{r^{2-2\nu}}
\end{align*}
Then $v$ extends continuously over $0$ and, moreover, $v\in W^{1,p}(B_{r_0},\bR^{2n})$.
\end{Lem}

\begin{proof}
By a standard argument involving Morrey's inequality (\ref{eqnMorrey}),
$v$ is uniformly Hölder continuous on the punctured disc.
Since $\bR^{2n}$ is complete, this shows that $v$ extends continuously over zero.
Finally, the weak first derivatives of $v$ exist on $B_{r_0}$ and agree with the strong first derivatives on $B_{r_0}\setminus\{0\}$.
\end{proof}

\begin{proof}[Proof of Thm. \ref{thmRemovalOfSingularities}]
The removable singularity theorem for the underlying holomorphic curve is a classical result
(cf. Thm. 4.1.2 in \cite{MS04}), and thus we may assume that $\varphi$ extends smoothly over $0$.
It remains to show the statement for the section $\psi$.
Consider any local holomorphic coordinates such that $p$ is mapped to $0\in\bR^2$ and
$(\varphi,\psi)$ is identified with a local holomorphic supercurve.
By Lem. \ref{lemSuperEnergyLocal}, it has finite
local super energy. Therefore, the hypotheses of Prp. \ref{prpRemovalOfSingularities} are
satisfied and, using the resulting estimate, we may apply Lem. \ref{lemHoelderContinuous},
by which $\psi$ extends over $0$ to a map of class $W^{1,p}$. By elliptic regularity
(Prp. 3.7 in \cite{Gro11b}), it follows that this extension
is smooth.
\end{proof}

\section{Bubbling}
\label{secBubbling}

In this section, we study the bubbling phenomenon for sequences of holomorphic supercurves
with domain $\Sigma=S^2$ and holomorphic line bundles $L_{-2}=T^*S^2$ and $L_{-1}=S^+$,
whose super energy is uniformly bounded above.
$(X,\omega)$ continues to be a compact symplectic manifold, and we let $\mA:=\mA(GL(X))$
and $\mJ:=\mJ(X,\omega)$ denote the respective spaces of connections and $\omega$-tame almost complex
structures on $X$. Moreover, we fix $(A,J)\in\mA\times\mJ$.
For the cases $d=\deg L_d<0$ of interest, we recall Rem. 2.3 of \cite{Gro11b}:

\begin{Lem}
\label{lemGhostBubble}
Let $(\varphi,\psi):(\Sigma,L)\rightarrow X$ be a holomorphic supercurve such that $\varphi$
is constant. Then, in the cases $\deg L<0$, $\psi\equiv 0$ vanishes identically.
\end{Lem}

By our first result, Prp. \ref{prpPsiLInfinityBounded}, which follows from the first
mean value inequality in the last section and a compactness result for holomorphic supercurves
with bounded $L^p$-norms for $p>2$ from \cite{Gro11b},
the emergence of bubbling points is solely determined by the underlying
sequence of holomorphic curves. We may thus follow the classical treatment of bubbling,
with one major exception: For the proof of Gromov
compactness in the next section, we need the super energy, not just the classical energy,
concentrated in a bubbling point to coincide with the super energy of the corresponding bubble.
The proof, concerning the sectional part of the super energy in the case $L=L_{-1}$, is provided by
Prp. \ref{prpSuperBubblesEnergy} below.
Our results are largely based on conformal invariance of the super energy.
Throughout this chapter, we let $G$ denote the group of Möbius transformations and,
with the notation from the previous section, we abbreviate
$\psi\circ m:=\tpsi=\Phi_m^{-1}\circ\psi\circ m$ for $m\in G$.

\begin{Prp}[Compactness, Prp. 3.8 in \cite{Gro11b}]
\label{prpPhiPsiConverging}
Let $(A,J)$ be a connection and an almost complex structure on $X$, both of regularity class $C^{l+1}$,
and let $(A^{\nu},J^{\nu})$ be a sequence of such objects that converges to $(A,J)$ in the
$C^{l+1}$-topology. Let $j^{\nu}$ be a sequence of complex structures on $\Sigma$
converging to $j$ in the $C^{\infty}$-topology, and let $U^{\nu}\subseteq\Sigma$ be an
increasing sequence of open sets whose union is $\Sigma$. Let $(\varphi^{\nu},\psi^{\nu})$
be a sequence of $(A^{\nu},J^{\nu})$-holomorphic supercurves
\begin{align*}
\varphi^{\nu}\in W^{1,p}(U^{\nu},X)\;,\qquad
\psi^{\nu}\in W^{1,p}(U^{\nu},L\otimes_{J^{\nu}}(\varphi^{\nu})^*TX)
\end{align*}
such that $p>2$ and assume that, for every compact set $Q\subseteq\Sigma$,
there exists a compact set $K\subseteq X$ and a constant $c>0$ such that
\begin{align*}
\norm{d\varphi^{\nu}}_{Q,p}\leq c\;,\qquad\varphi^{\nu}(Q)\subseteq K\;,\qquad
\norm{\psi^{\nu}}_{Q,p}\leq c
\end{align*}
for $\nu$ sufficiently large.
Then there exists a subsequence of $(\varphi^{\nu},\psi^{\nu})$, which converges in
the $C^l$-topology on every compact subset of $\Sigma$.
\end{Prp}

\begin{Prp}
\label{prpPsiLInfinityBounded}
Let $(A^{\nu},J^{\nu})\in\mA\times\mJ$ be a sequence that converges to $(A,J)$ in the
$C^{\infty}$-topology and $(\varphi^{\nu},\psi^{\nu}):(U,L_d)\rightarrow X$ be a sequence of
$(A^{\nu},J^{\nu})$-holomorphic supercurves on an open subset $U\subseteq S^2$.
Moreover, assume that the sequence $\varphi^{\nu}$ converges to a holomorphic curve
$\varphi:U\rightarrow X$ in the $C^{\infty}$-topology on a compact subset $K\subseteq U$,
and that the energy of $\psi^{\nu}$ has a uniform upper bound:
\begin{align*}
\sup_{\nu}E(\psi^{\nu},U)<\infty
\end{align*}
Then, in the cases $d=-2$ and $d=-1$, the sup norm of $\psi^{\nu}$ is uniformly bounded on $K$:
\begin{align*}
\sup_{\nu}\norm{\psi^{\nu}}_{K,\infty}<\infty
\end{align*}
In particular, there is a subsequence, also labelled $\nu$, such that $\psi^{\nu}$ converges
in the $C^{\infty}$-topology on $K$.
\end{Prp}

\begin{proof}
For $z\in K$, we choose a sufficiently small positive integer $r_z$ such that
$\varphi^{\nu}(B_{2r_z}(z))$ is contained, for large $\nu$, in a coordinate chart of $X$;
$L_d$ is then trivial on $B_{2r_z}(z)$.
Since $K$ is compact, there are finitely many $z_i\in K$ such that the balls
$B_{r_i}(z_i)$ with $r_i:=r_{z_i}$ cover $K$. Restricted to $B_{2r_i}(z_i)$,
we may identify $(\varphi^{\nu},\psi^{\nu})$ with a local holomorphic supercurve
and, moreover, identify the bundle norm on $L_d\otimes_J(\varphi^{\nu})^*TX$ with
the standard norm on $\bR^{2n}$ by Lem. \ref{lemLocalGlobalEstimate}, with constants
that depend on the trivialisations and $J^{\nu}$ and, by convergence of $J^{\nu}$,
may hence be uniformly estimated. By the resulting inequality
\begin{align*}
\norm{\psi^{\nu}}_{K,\infty}\leq C\sum_i\norm{\psi^{\nu}}_{K\cap B_{r_i}(z_i),\infty}
\end{align*}
it suffices to show that the ($\bR^{2n}$)-sup norm of $\psi^{\nu}$ on $K\cap B_{r_i}(z_i)$
is uniformly bounded for every $i$.

Since, by assumption, the super energy of $(\varphi^{\nu},\psi^{\nu})$ is uniformly bounded above
and the sequences $\varphi^{\nu}$ and $(A^{\nu},J^{\nu})$ are convergent, there is a uniform upper
bound for the local super energy by Lem. \ref{lemSuperEnergyLocal}. In particular, we may assume that
$\int_{B_{2r_i}(z_i)}\abs{\psi^{\nu}}^2<\frac{\pi}{16}$ holds for all $i$ and $\nu$,
for otherwise replace $\psi^{\nu}$ by $A\cdot\psi^{\nu}$ where $A$ is a sufficiently
small constant. The hypotheses of the mean value inequality in Lem. \ref{lemMVI1} are then satisfied
with $0$ and $r$ replaced by $z\in B_{r_i}(z_i)$ and $r_i$, respectively, and we obtain
a constant $a^{\nu}$, depending on $\varphi^{\nu}$ and $(A^{\nu},J^{\nu})$ such that
the following estimate holds.
\begin{align*}
\abs{\psi^{\nu}(z)}^2\leq\frac{a^{\nu}r_i^2}{8}+\frac{8}{\pi r_i^2}\int_{B_{r_i}(z)}\abs{\psi^{\nu}}^2
<\frac{a^{\nu}r_i^2}{8}+\frac{1}{2r_i^2}
\end{align*}
There is a uniform upper bound for $a^{\nu}$, again by the convergence of the sequences
$\varphi^{\nu}$ and $(A^{\nu},J^{\nu})$ and, as a consequence,
$\abs{\psi^{\nu}(z)}$ is uniformly bounded for every $z\in B_{r_i}(z_i)$,
which was to show.

The last part of the statement is then an immediate consequence of Prp. \ref{prpPhiPsiConverging}.
\end{proof}

The next lemma will be useful in Sec. \ref{subsecGromovTopology}.
Note that, despite most results in this chapter, it does not depend on the degree of $L$.
Moreover, there is no need to pass to a subsequence here.

\begin{Lem}
\label{lemUniformImpliesCInfty}
Let $(A^{\nu},J^{\nu})\in\mA\times\mJ$ be a sequence that converges to $(A,J)$ in the
$C^{\infty}$-topology.
Let $\Omega\subseteq\bC$ be an open set and $(\varphi^{\nu},\psi^{\nu}):(\Omega,L_d)\rightarrow X$
be a sequence of $(A^{\nu},J^{\nu})$-holomorphic supercurves. Moreover, suppose that
$(\varphi^{\nu},\psi^{\nu})$ converges uniformly (i.e. in the $C^0$-toplogy) to a continuous pair
$(\varphi,\psi):(\Omega,L_d)\rightarrow X$. Then $(\varphi,\psi)$
is an $(A,J)$-holomorphic supercurve, and $(\varphi^{\nu},\psi^{\nu})$ converges to
$(\varphi,\psi)$ in the $C^{\infty}$-topology on every compact subset of $\Omega$.
\end{Lem}

\begin{proof}
The statement for $\varphi^{\nu}$ is the content of \cite{MS04}, Lem. 4.6.6,
which follows as a consequence of elliptic bootstrapping and a conformal rescaling
argument.

By hypothesis, we have $\sup_{\nu}\norm{\psi^{\nu}}_{U,\infty}<\infty$,
and we use the compactness result from Prp. \ref{prpPhiPsiConverging}:
If the statement does not hold, then there is a subsequence $\nu_j$, an integer $l$ and $\delta>0$
such that $d_{C^l}(\psi^{\nu_j},\psi)\geq\delta$ for all $j$.
On the other hand, there is a further subsequence, also denoted $\nu_j$, such that
$\psi^{\nu_j}$ converges to $\psi$ in the $C^l$-topology, which is a contradiction.
\end{proof}

\begin{Thm}[Convergence Modulo Bubbling]
\label{thmConvergenceModBubbling}
Let $(A^{\nu},J^{\nu})\in\mA\times\mJ$ be a sequence that converges to $(A,J)$ in the
$C^{\infty}$-topology and $(\varphi^{\nu},\psi^{\nu}):(S^2,L_d)\rightarrow X$ be a sequence of
$(A^{\nu},J^{\nu})$-holomorphic supercurves
such that
\begin{align*}
\sup_{\nu}E(\varphi^{\nu},\psi^{\nu})<\infty
\end{align*}
Then, in the cases $d=-2$ and $d=-1$, there exist a subsequence (still denoted
$(\varphi^{\nu},\psi^{\nu})$), an $(A,J)$-holomorphic supercurve $(\varphi,\psi):(S^2,L_d)\rightarrow X$
and a finite set $Z=\{z_1,\ldots,z_l\}\subseteq\Sigma$ such that the following holds.
\begin{enumerate}
\renewcommand{\labelenumi}{(\roman{enumi})}
\item $(\varphi^{\nu},\psi^{\nu})$ converges to $(\varphi,\psi)$ in the $C^{\infty}$-topology
on all compact subsets of $S^2\setminus Z$.
\item For every $j$ and every $\varepsilon>0$ such that $B_{\varepsilon}(z_j)\cap Z=\{z_j\}$, the limits
\begin{align*}
m_{\varepsilon}^{\varphi}(z_j):=\lim_{\nu\rightarrow\infty}E(\varphi^{\nu},B_{\varepsilon}(z_j))\;,\qquad
m_{\varepsilon}^{\psi}(z_j):=\lim_{\nu\rightarrow\infty}E(\psi^{\nu},B_{\varepsilon}(z_j))
\end{align*}
exist and are continuous functions of $\varepsilon$, and
\begin{align*}
m^{\varphi}(z_j):=\lim_{\varepsilon\rightarrow 0}m_{\varepsilon}^{\varphi}(z_j)\geq\hbar\;,\qquad
m^{\psi}(z_j):=\lim_{\varepsilon\rightarrow 0}m_{\varepsilon}^{\psi}(z_j)\geq 0
\end{align*}
where $\hbar>0$ denotes the minimal classical energy for nonconstant $J$-holomorphic spheres.
\item For every compact subset $K\subseteq S^2$ with $Z\subseteq\mathrm{int}(K)$,
\begin{align*}
E(\varphi,K)+\sum_{j=1}^lm^{\varphi}(z_j)&=\lim_{\nu\rightarrow\infty}E(\varphi^{\nu},K)\\
E(\psi,K)+\sum_{j=1}^lm^{\psi}(z_j)&=\lim_{\nu\rightarrow\infty}E(\psi^{\nu},K)
\end{align*}
\end{enumerate}
\end{Thm}

\begin{proof}
By an inductive argument over singular points $z_j\in S^2$, one constructs
a finite set $Z$ such that (a subsequence of) $\varphi^{\nu}$ converges in $C^{\infty}_{\mathrm{loc}}$
to a holomorphic curve $\varphi$ on compact sets $K\subseteq S^2\setminus Z$.
This is detailed in the proof of Thm. 4.6.1 in \cite{MS04}.
By Prp. \ref{prpPsiLInfinityBounded}, there is a (further) subsequence such that also
$\psi^{\nu}$ converges in the $C^{\infty}$-topology on $K$ to a map $\psi:S^2\setminus Z\rightarrow X$,
such that $(\varphi,\psi)$ is a holomorphic supercurve.
By the removable singularity Theorem \ref{thmRemovalOfSingularities},
$(\varphi,\psi)$ extends to a holomorphic supercurve $(S^2,L_d)\rightarrow X$.
This concludes the proof of (i).
The proofs of (ii) and (iii) are almost verbatim to the corresponding steps in the
proof of Thm. 4.6.1 in \cite{MS04}.
\end{proof}

\subsection{Conservation of Super Energy}
\label{subsecEnergyIdentities}

Having established convergence modulo bubbling, we examine in the remainder of this section
the actual bubbling off of holomorphic superspheres. In particular, we show that the
rescaling sequences can be chosen so that there is no loss of super energy or,
to be more precise, that the super energy concentrated in a bubbling point
coincides with the energy of the corresponding bubble, cf. identities (iii)
and (v) in Prp. \ref{prpSuperBubbles} and Prp. \ref{prpSuperBubblesEnergy} below.
Our treatment resembles the analysis of the classical bubbling as described
in Chp. 4 of \cite{MS04}, using in addition a variant of an estimate from
\cite{CJLW05} in a quite intricate context.

\begin{Prp}
\label{prpSuperBubbles}
Let $(A^{\nu},J^{\nu})\in\mA\times\mJ$ be a sequence that converges to $(A,J)$ in the
$C^{\infty}$-topology. Fix a point $z_0\in\bC$ and a number $r>0$. Let
$(\varphi^{\nu},\psi^{\nu}):(\overline{B_r(z_0)},L_d)\rightarrow X$ be
a sequence of $(A^{\nu},J^{\nu})$-holomorphic supercurves and
$(\varphi,\psi):(\overline{B_r(z_0)},L_d)\rightarrow X$ be an $(A,J)$-holomorphic supercurve
such that the following holds.
\begin{enumerate}
\renewcommand{\labelenumi}{(\alph{enumi})}
\item $(\varphi^{\nu},\psi^{\nu})$ converges to $(\varphi,\psi)$ in the $C^{\infty}$-topology on every
compact subset of $\overline{B_r(z_0)}\setminus\{z_0\}$.
\item The limit $m_0^{\varphi}:=\lim_{\varepsilon\rightarrow 0}\lim_{\nu\rightarrow\infty}
E(\varphi^{\nu},B_{\varepsilon}(z_0))$ exists and is positive.
\item The limit $m_0^{\psi}:=\lim_{\varepsilon\rightarrow 0}\lim_{\nu\rightarrow\infty}
E(\psi^{\nu},B_{\varepsilon}(z_0))$ exists.
\end{enumerate}
Then, in the cases $d=-2$ and $d=-1$, there exist a subsequence, still denoted by
$(\varphi^{\nu},\psi^{\nu})$, a sequence of Möbius transformations $m^{\nu}\in G$,
an $(A,J)$-holomorphic supercurve $(\tvarphi,\tpsi):(S^2,L_d)\rightarrow X$,
and finitely many distinct points $z_1,\ldots,z_l,z_{\infty}\in S^2$ such that the
following holds.
\begin{enumerate}
\renewcommand{\labelenumi}{(\roman{enumi})}
\item $m^{\nu}$ converges to $z_0$ in the $C^{\infty}$-topology on every compact subset
of $S^2\setminus\{z_{\infty}\}$.
\item The sequence
$(\tvarphi^{\nu}:=\varphi^{\nu}\circ m^{\nu},\tpsi^{\nu}:=\psi^{\nu}\circ m^{\nu})$
converges to $(\tvarphi,\tpsi)$ in the $C^{\infty}$-topology on every compact subset of
$S^2\setminus\{z_1,\ldots,z_l,z_{\infty}\}$, and the limits
\begin{align*}
m_j^{\varphi}&:=\lim_{\varepsilon\rightarrow 0}\lim_{\nu\rightarrow\infty}E(\tvarphi^{\nu},B_{\varepsilon}(z_j))\\
m_j^{\psi}&:=\lim_{\varepsilon\rightarrow 0}\lim_{\nu\rightarrow\infty}E(\tpsi^{\nu},B_{\varepsilon}(z_j))
\end{align*}
exist for $j=1,\ldots,l$ and, moreover, $m_j^{\varphi}>0$ is positive.
\item $E(\tvarphi,S^2)+\sum_{j=1}^lm_j^{\varphi}=m_0^{\varphi}$.
\item If $\tvarphi$ is constant then $l\geq 2$ and $\tpsi\equiv 0$.
\end{enumerate}
\end{Prp}

\begin{proof}
The extension of the proof of Prp. 4.7.1 in \cite{MS04} to the present
situation is straightforward. We summarise the most important steps.
First we may, without loss of generality, assume that $z_0=0$
and that each function $z\mapsto\abs{d\varphi^{\nu}}$ assumes its maximum at $z=0$.
Step 2 of the original proof then yields a sequence $\delta^{\nu}>0$
converging to zero such that $m^{\nu}:=\delta^{\nu}\cdot$ satisfies (i) with $z_{\infty}:=\infty$.
The super energy of the sequence
\begin{align*}
(\tvarphi^{\nu},\tpsi^{\nu}):(B_{r/\delta^{\nu}},L_d)\rightarrow X\;,\qquad
\tvarphi^{\nu}(z):=\varphi^{\nu}(\delta^{\nu}\cdot z)\;,\qquad
\tpsi^{\nu}(z):=\psi^{\nu}(\delta^{\nu}\cdot z)
\end{align*}
is bounded since that of the original sequence $(\varphi^{\nu},\psi^{\nu})$ is.
Therefore, the hypotheses of Thm. \ref{thmConvergenceModBubbling} are
satisfied, providing a subsequence of $(\tvarphi^{\nu},\tpsi^{\nu})$ such that
(ii) holds for some finite set $Z=\{z_1,\ldots,z_l,z_{\infty}\}\subseteq S^2$.
The proof of (iii) is based on the following identity.
\begin{align}
\label{eqnBubbleEnergyIdentity}
\lim_{R\rightarrow\infty}\lim_{\nu\rightarrow\infty}E(\varphi^{\nu},B_{R\delta^{\nu}})=m_0^{\varphi}
\end{align}
Finally, the second part of (iv) is implied by Lem. \ref{lemGhostBubble}.
\end{proof}

By the next proposition, proved as the first part of Prp. 4.7.2 in \cite{MS04},
the classical bubbles connect.

\begin{Prp}
\label{prpBubblesConnect}
Let $J^{\nu}\in\mJ$, $z_0\in\bC$ and $\varphi,\varphi^{\nu}:\overline{B_r(z_0)}\rightarrow X$
be as in the hypotheses of Prp. \ref{prpSuperBubbles} and suppose that $m^{\nu}\in G$,
$\tvarphi:S^2\rightarrow X$ and $z_1,\ldots,z_l,z_{\infty}\in S^2$ satisfy the assertions
(i)-(iii) of Prp. \ref{prpSuperBubbles}. Then
$\varphi(z_0)=\tvarphi(z_{\infty})$ holds.\\
Moreover, for every $\kappa>0$, there exist constants $\gamma>0$ and $\nu_0\in\bN$ such that
\begin{align*}
\scal[d_{S^2}]{z}{z_0}+\scal[d_{S^2}]{(m^{\nu})^{-1}(z)}{z_{\infty}}<\gamma\implies
\scal[d_{g_J}]{\varphi^{\nu}(z)}{\varphi(z_0)}<\kappa
\end{align*}
for every integer $\nu\geq\nu_0$ and every $z\in S^2$.
\end{Prp}

Note that we do not claim an analogous statement for holomorphic supercurves as a whole.
On the other hand, the energy identity (iii) in Prp. \ref{prpSuperBubbles} does have
a natural generalisation to the super energy as shown next.
We need the following lemma, which provides a back door for the case that the
$W^{1,\infty}$-norm of $\tJ$ cannot be estimated. We denote by $J_0$
a fixed complex structure on $\bR^{2n}$, e.g. the standard one.

\begin{Lem}
\label{lemDbarEstimate}
Let $U'\subseteq U\subseteq\bC$ and $V\subseteq\bR^{2n}$ be open sets such that
$\overline{U'}\subseteq U$ and $1<p<\infty$.
Then, for every constant $c_0>0$, there is a constant $c>0$ such that the following holds. Assume
$J\in W^{1,\infty}(V,\bR^{2n\times 2n})$ satisfies $J^2=-\id$ and $\norm{J}_{V,1,\infty}\leq c_0$
and consider the map
\begin{align*}
\phi:V\rightarrow\bR^{2n\times 2n}\;,\qquad x\mapsto\phi_x\;,\qquad
\phi_x(v):=\frac{1}{2}(v-J_0\cdot J_x\cdot v)
\end{align*}
which satisfies $\phi_x\circ J_x=J_0\circ\phi_x$ and is bijective if $J_x$ is
sufficiently close to $J_0$.
Then, for every $u\in W^{1,p}_{\mathrm{loc}}(U,\bR^{2n})$ and $\varphi\in C^1(U,V)$,
the following estimate holds, abbreviating $\tJ:=J\circ\varphi$.
\begin{align*}
\norm{(\phi\circ\varphi)\cdot u}_{U',1,p}\leq c\left(\norm{\abs{d\varphi}\abs{u}}_{U,0,p}
+\norm{\partial_su+\tJ\partial_tu}_{U,0,p}+\norm{u}_{U,0,p}\right)
\end{align*}
\end{Lem}

\begin{proof}
The hypotheses of Lem. \ref{lemEllipticBootstrappingWeak} are satisfied for $J_0$,
and the resulting inequality can be further estimated as follows.
\begin{align*}
&\norm{(\phi\circ\varphi)\cdot u}_{U',1,p}\\
&\qquad\leq c\left(\norm{(\partial_s+J_0\partial_t)((\phi\circ\varphi)\cdot u)}_{U,0,p}
+\norm{(\phi\circ\varphi)\cdot u}_{U,0,p}\right)\\
&\qquad=c\left(\norm{(\partial_s+J_0\partial_t)(\phi\circ\varphi)\cdot u
+(\phi\circ\varphi)\cdot(\partial_s+\tJ\partial_t)u}_{U,0,p}+\norm{(\phi\circ\varphi)\cdot u}_{U,0,p}\right)\\
&\qquad\leq c\norm{\phi}_{V,1,\infty}\left(\norm{(\abs{\partial_s\varphi}
+\abs{\partial_t\varphi})\cdot\abs{u}}_{U,0,p}
+\norm{\partial_su+\tJ\partial_tu}_{U,0,p}+\norm{u}_{U,0,p}\right)\\
&\qquad\leq c_2\left(\norm{\abs{d\varphi}\abs{u}}_{U,0,p}
+\norm{\partial_su+\tJ\partial_tu}_{U,0,p}+\norm{u}_{U,0,p}\right)
\end{align*}
where $c_2:=c\norm{\phi}_{V,1,\infty}$ depends only on $c_0$ (and $J_0$).
\end{proof}

\begin{Prp}
\label{prpSuperBubblesEnergy}
In the situation of Prp. \ref{prpSuperBubbles}, let $d=-1$. Then, additionally, the energy identity
\begin{itemize}
\item[(v)] $E(\tpsi,S^2)+\sum_{j=1}^lm_j^{\psi}=m_0^{\psi}$
\end{itemize}
holds true.
\end{Prp}

\begin{proof}
We continue the proof of Prp. \ref{prpSuperBubbles} with the same notations.

\emph{Step 1: Let $A(r,R):=B_R\setminus\overline{B_r}$ denote the (open) annulus with radii $0<r<R$.
We prove the identity}
\begin{align*}
\lim_{\varepsilon\rightarrow 0}\lim_{\nu\rightarrow\infty}
\int_{A\left(\frac{\delta^{\nu}}{\varepsilon},\varepsilon\right)}\abs{\psi^{\nu}}^4\dvol_{S^2}=0
\end{align*}
Let $0\leq\chi^{\nu,\varepsilon}\leq 1$ be a cutoff function such that
\begin{align*}
\chi^{\nu,\varepsilon}\in C^{\infty}\left(A\left(\frac{\delta^{\nu}}{2\varepsilon},2\varepsilon\right)\right)\;
\textrm{with compact support}\;&,\qquad
\chi^{\nu,\varepsilon}\equiv 1\quad\mathrm{in}\quad A\left(\frac{\delta^{\nu}}{\varepsilon},\varepsilon\right)\\
\abs{d\chi^{\nu,\varepsilon}}\leq\frac{C}{\varepsilon}\quad\mathrm{in}\quad A(\varepsilon,2\varepsilon)
\;&,\qquad\abs{d\chi^{\nu,\varepsilon}}\leq\frac{C\varepsilon}{\delta^{\nu}}\quad\mathrm{in}\quad
A\left(\frac{\delta^{\nu}}{2\varepsilon},\frac{\delta^{\nu}}{\varepsilon}\right)
\end{align*}
for some constant $C>0$ (independent of $\nu$ and $\varepsilon$).
For an explicit construction of such a function $\chi^{\nu,\varepsilon}$, consult Sec. 2.18 in \cite{Alt02}.
We abbreviate
\begin{align*}
A_1^{\nu,\varepsilon}:=A\left(\frac{\delta^{\nu}}{\varepsilon},\varepsilon\right)\;,\qquad
A_2^{\nu,\varepsilon}:=A\left(\frac{\delta^{\nu}}{2\varepsilon},2\varepsilon\right)
\end{align*}

Choose $\kappa>0$ sufficiently small such that $B_{\kappa}(\varphi(0))\subseteq X$ is
contained in a coordinate chart of $X$.
We aim at applying Lem. \ref{lemDbarEstimate} with $V$, $J_0$ and $J$ replaced by
$B_{\kappa}(\varphi(0))$, $J(\varphi(0))$ and $J^{\nu}$, respectively.
Upon making $\kappa$ smaller if necessary we may assume, by convergence of $J^{\nu}$,
that $J^{\nu}(p)$ is sufficiently close to $J(\varphi(0))$ such that $\phi^{\nu}_p$
(as defined in Lem. \ref{lemDbarEstimate}) is bijective
for all $p\in B_{\kappa}(\varphi(0))$ and for large $\nu$.
Let $\gamma>0$ be the resulting constant from the second part of Prp. \ref{prpBubblesConnect},
depending on $\kappa$. By definition, points $z\in A_2^{\nu,\varepsilon}$ satisfy
$z<2\varepsilon$ and $\frac{1}{2\varepsilon}<(\delta^{\nu})^{-1}z$.
Therefore, there is a (sufficiently small) constant $\varepsilon_0>0$ such that,
for all $\varepsilon\leq\varepsilon_0$ and for all $z\in A_2^{\nu,\varepsilon}$, we yield
\begin{align*}
\scal[d_{S^2}]{z}{0}+\scal[d_{S^2}]{(\delta^{\nu})^{-1}(z)}{\infty}<\gamma
\end{align*}

By Prp. \ref{prpBubblesConnect}, it then follows that $\varphi^{\nu}(z)\in B_{\kappa}(\varphi(0))$
for all $z\in A_2^{\nu,\varepsilon}$, provided that $\varepsilon\leq\varepsilon_0$ and $\nu$ is sufficiently
large, what we shall assume in the following.
By construction of $\kappa$, the bundle $L_d\otimes_{J^{\nu}}(\varphi^{\nu})^*TX$,
restricted to $A_2^{\nu,\varepsilon}$, is thus trivial and, by Lem. \ref{lemLocalGlobalEstimate},
the bundle norm (induced by the Hermitian metric from (\ref{eqnH2}))
is equivalent to the standard norm on $\bR^{2n}$. Moreover, the constants by which these
two norms are estimated against each other depend only on $J^{\nu}$,
and thus have a uniform upper bound. Restricted to $A_2^{\nu,\varepsilon}$, we shall therefore, in the
following, blur the distinction between the bundle norm and the standard norm on $\bR^{2n}$.

We estimate
\begin{align*}
\norm{\psi^{\nu}}_{A_1^{\nu,\varepsilon},0,4}
&\leq\norm{\chi^{\nu,\varepsilon}\psi^{\nu}}_{B_{2\varepsilon_0}(0),0,4}\\
&=\norm{((\phi^{\nu})^{-1}\circ\varphi^{\nu})\cdot
(\phi^{\nu}\circ\varphi^{\nu})\cdot\chi^{\nu,\varepsilon}\psi^{\nu}}_{B_{2\varepsilon_0},0,4}\\
&\leq\norm{(\phi^{\nu})^{-1}}_{B_{\kappa}(\varphi(0)),0,\infty}
\norm{(\phi^{\nu}\circ\varphi^{\nu})\cdot\chi^{\nu,\varepsilon}\psi^{\nu}}_{B_{2\varepsilon_0}(0),0,4}\\
&\leq C_1\norm{(\phi_i^{\nu}\circ\varphi^{\nu})\cdot\chi^{\nu,\varepsilon}\psi^{\nu}}_{B_{2\varepsilon_0}(0),1,\frac{4}{3}}
\end{align*}
where, for the last inequality, we used boundedness of the norm involving $\phi^{\nu}$ (by
convergence of $J^{\nu}$) as well as the Sobolev embedding (\ref{eqnSobolev})
with $k=1$ and $p=\frac{4}{3}$. By Lem. \ref{lemDbarEstimate}, with input data as above
as well as $U':=B_{2\varepsilon_0}(0)$ and $U:=B_{3\varepsilon_0}(0)$, we further obtain
\begin{align*}
&\norm{\psi^{\nu}}_{A_1^{\nu,\varepsilon},0,4}\\
&\quad\leq C_2\left(\norm{\abs{d\varphi^{\nu}}\abs{\chi^{\nu,\varepsilon}\psi^{\nu}}}_{B_{3\varepsilon_0},0,\frac{4}{3}}
+\norm{(\partial_s+\tJ^{\nu}\partial_t)(\chi^{\nu,\varepsilon}\psi^{\nu})}_{B_{3\varepsilon_0},0,\frac{4}{3}}
+\norm{\chi^{\nu,\varepsilon}\psi^{\nu}}_{B_{3\varepsilon_0},0,\frac{4}{3}}\right)\\
&\quad=C_2\left(\norm{\abs{d\varphi^{\nu}}\abs{\chi^{\nu,\varepsilon}\psi^{\nu}}}_{A_2^{\nu,\varepsilon},0,\frac{4}{3}}
+\norm{(\partial_s+\tJ^{\nu}\partial_t)(\chi^{\nu,\varepsilon}\psi^{\nu})}_{A_2^{\nu,\varepsilon},0,\frac{4}{3}}
+\norm{\chi^{\nu,\varepsilon}\psi^{\nu}}_{A_2^{\nu,\varepsilon},0,\frac{4}{3}}\right)\\
&\quad\leq C_2\left(\norm{d\varphi^{\nu}}_{A_2^{\nu,\varepsilon},0,2}\norm{\chi^{\nu,\varepsilon}\psi^{\nu}}_{A_2^{\nu,\varepsilon},0,4}
+\norm{(\partial_s+\tJ^{\nu}\partial_t)(\chi^{\nu,\varepsilon}\psi^{\nu})}_{A_2^{\nu,\varepsilon},0,\frac{4}{3}}\right.\\
&\qquad\qquad\qquad\qquad\qquad\qquad\qquad\qquad\qquad\qquad\qquad\left.+\norm{\chi^{\nu,\varepsilon}}_{A_2^{\nu,\varepsilon},0,2}
\norm{\psi^{\nu}}_{A_2^{\nu,\varepsilon},0,4}\right)
\end{align*}
Here, the equality is due to $\chi^{\nu,\varepsilon}\equiv 0$ outside $A_2^{\nu,\varepsilon}$, and the
last estimate is H\"older's inequality in the form (\ref{eqnHoelder43}).
In particular, the constants $C_1$ and $C_2$ are independent of $\nu$ and $\varepsilon$.

We must show that every summand in the previous estimate vanishes in the limit.
For the third term, this is clear
for, by definition, $E(\psi^{\nu},B_r)=\norm{\psi^{\nu}}_{B_r,0,4}$ is uniformly bounded and
$\norm{\chi^{\nu,\varepsilon}}_{A_2^{\nu,\varepsilon},0,2}\leq\norm{1}_{B_{2\varepsilon},0,2}$
goes to zero as $\varepsilon\rightarrow 0$.
As for the first summand, the second factor $\norm{\chi^{\nu,\varepsilon}\psi^{\nu}}_{A_2^{\nu,\varepsilon},0,4}$
is uniformly bounded whereas, by (\ref{eqnBubbleEnergyIdentity}), the first factor
\begin{align*}
\lim_{\varepsilon\rightarrow 0}\lim_{\nu\rightarrow\infty}
\norm{d\varphi^{\nu}}_{A_2^{\nu,\varepsilon},0,2}^2
&=\lim_{\varepsilon\rightarrow 0}\lim_{\nu\rightarrow\infty}\norm{d\varphi^{\nu}}_{B_{2\varepsilon},2}^2
-\lim_{\varepsilon\rightarrow 0}\lim_{\nu\rightarrow\infty}\norm{d\varphi^{\nu}}_{B_{\frac{\delta^{\nu}}{2\varepsilon}},2}^2\\
&=\lim_{\varepsilon\rightarrow 0}\lim_{\nu\rightarrow\infty}\norm{d\varphi^{\nu}}_{B_{\varepsilon},2}^2
-\lim_{R\rightarrow\infty}\lim_{\nu\rightarrow\infty}\norm{d\varphi^{\nu}}_{B_{R\delta^{\nu}},2}^2
=0
\end{align*}
vanishes in the limit.

Hence, it remains to consider the second term in the previous estimate.
By (\ref{eqnLocalHolomorphicSupercurve}) there are $\tD^{\nu}$,
depending on $d\varphi^{\nu}$ and $(A^{\nu},J^{\nu})$, such that
\begin{align*}
\norm{(\partial_s+\tJ^{\nu}\partial_t)(\chi^{\nu,\varepsilon}\psi^{\nu})}_{A_2^{\nu,\varepsilon},0,\frac{4}{3}}
&=\norm{\chi^{\nu,\varepsilon}((\partial_s+\tJ^{\nu}\partial_t)\psi^{\nu})
+(\dbar\chi^{\nu,\varepsilon})\psi^{\nu}}_{A_2^{\nu,\varepsilon},0,\frac{4}{3}}\\
&=\norm{-\chi^{\nu,\varepsilon}\tD^{\nu}\psi^{\nu}+(\dbar\chi^{\nu,\varepsilon})\psi^{\nu}}_{A_2^{\nu,\varepsilon},0,\frac{4}{3}}\\
&\leq C_4\norm{\abs{d\varphi^{\nu}}\abs{\psi^{\nu}}}_{A_2^{\nu,\varepsilon},0,\frac{4}{3}}
+\norm{(\dbar\chi^{\nu,\varepsilon})\psi^{\nu}}_{A_2^{\nu,\varepsilon},0,\frac{4}{3}}\\
&\leq C_5\norm{d\varphi^{\nu}}_{A_2^{\nu,\varepsilon},0,2}\norm{\psi^{\nu}}_{A_2^{\nu,\varepsilon},0,4}
+\norm{(\dbar\chi^{\nu,\varepsilon})\psi^{\nu}}_{A_2^{\nu,\varepsilon},0,\frac{4}{3}}
\end{align*}
with $\dbar:=\partial_s+i\partial_t$ using that, by definition, $\tJ^{\nu}$ acts on $\psi^{\nu}$ via $i$.
The first term in the last estimate is already known to vanish in the limit, and we estimate the second
one as follows.
\begin{align*}
&\norm{(\dbar\chi^{\nu,\varepsilon})\psi^{\nu}}_{A_2^{\nu,\varepsilon},0,\frac{4}{3}}\\
&\qquad=\norm{(\dbar\chi^{\nu,\varepsilon})\psi^{\nu}}
_{A\left(\frac{\delta^{\nu}}{2\varepsilon},\frac{\delta^{\nu}}{\varepsilon}\right),0,\frac{4}{3}}
+\norm{(\dbar\chi^{\nu,\varepsilon})\psi^{\nu}}_{A(\varepsilon,2\varepsilon),0,\frac{4}{3}}\\
&\qquad\leq\norm{\dbar\chi^{\nu,\varepsilon}}_{A\left(\frac{\delta^{\nu}}{2\varepsilon},\frac{\delta^{\nu}}{\varepsilon}\right),0,2}
\norm{\psi^{\nu}}_{A\left(\frac{\delta^{\nu}}{2\varepsilon},\frac{\delta^{\nu}}{\varepsilon}\right),0,4}
+\norm{\dbar\chi^{\nu,\varepsilon}}_{A(\varepsilon,2\varepsilon),0,2}
\norm{\psi^{\nu}}_{A(\varepsilon,2\varepsilon),0,4}\\
&\qquad=\norm{d\chi^{\nu,\varepsilon}}_{A\left(\frac{\delta^{\nu}}{2\varepsilon},\frac{\delta^{\nu}}{\varepsilon}\right),0,2}
\norm{\tpsi^{\nu}}_{A\left(\frac{1}{2\varepsilon},\frac{1}{\varepsilon}\right),0,4}
+\norm{d\chi^{\nu,\varepsilon}}_{A(\varepsilon,2\varepsilon),0,2}
\norm{\psi^{\nu}}_{A(\varepsilon,2\varepsilon),0,4}\\
&\qquad=:I+II
\end{align*}
For the last equation, we used
$\abs{\dbar\chi^{\nu,\varepsilon}}^2=\abs{\partial_s\chi^{\nu,\varepsilon}}^2
+\abs{\partial_t\chi^{\nu,\varepsilon}}^2=\abs{d\chi^{\nu,\varepsilon}}^2$
as well as conformal invariance (\ref{eqnSuperEnergyInvariance}) of the energy of $\psi^{\nu}$ which,
for the case $d=-1$ considered, coincides with the $L^4$-norm.

To proceed further, we use the estimates for the first derivative of $\chi^{\nu,\varepsilon}$, which form
part of its defining properties summarised above, to obtain that $I$ vanishes in the limit
by the estimates
\begin{align*}
\lim_{\varepsilon\rightarrow 0}\lim_{\nu\rightarrow\infty}I
\leq C_6\lim_{\varepsilon\rightarrow 0}\lim_{\nu\rightarrow\infty}
\left(\int_{A\left(\frac{\delta^{\nu}}{2\varepsilon},\frac{\delta^{\nu}}{\varepsilon}\right)}
\frac{\varepsilon^2}{(\delta^{\nu})^2}\right)^{\frac{1}{2}}
\norm{\tpsi^{\nu}}_{A\left(\frac{1}{2\varepsilon},\frac{1}{\varepsilon}\right),4}
\leq C_7\lim_{\varepsilon\rightarrow 0}\norm{\tpsi}_{A\left(\frac{1}{2\varepsilon},\frac{1}{\varepsilon}\right),4}=0
\end{align*}
for some constants $C_6,C_7>0$ and, similarly,
\begin{align*}
\lim_{\varepsilon\rightarrow 0}\lim_{\nu\rightarrow\infty}II
\leq C_8\lim_{\varepsilon\rightarrow 0}\lim_{\nu\rightarrow\infty}
\left(\int_{A(\varepsilon,2\varepsilon)}\frac{1}{\varepsilon^2}\right)^{\frac{1}{2}}
\norm{\psi^{\nu}}_{A(\varepsilon,2\varepsilon),4}
\leq C_9\lim_{\varepsilon\rightarrow 0}\norm{\psi}_{A(\varepsilon,2\varepsilon),4}=0
\end{align*}
for constants $C_8,C_9>0$.
This concludes the proof of Step 1.

\emph{Step 2: We prove (v).}\\
Using Step 1, we calculate
\begin{align*}
&\lim_{R\rightarrow\infty}\lim_{\nu\rightarrow\infty}\int_{B_{R\delta^{\nu}}}\abs{\psi^{\nu}}^4\dvol_{S^2}\\
&\qquad=\lim_{\varepsilon\rightarrow 0}\lim_{\nu\rightarrow\infty}
\int_{B_{\frac{\delta^{\nu}}{\varepsilon}}}\abs{\psi^{\nu}}^4\dvol_{S^2}\\
&\qquad=\lim_{\varepsilon\rightarrow 0}\lim_{\nu\rightarrow\infty}
\int_{B_{\varepsilon}}\abs{\psi^{\nu}}^4\dvol_{S^2}
-\lim_{\varepsilon\rightarrow 0}\lim_{\nu\rightarrow\infty}
\int_{A\left(\frac{\delta^{\nu}}{\varepsilon},\varepsilon\right)}\abs{\psi^{\nu}}^4\dvol_{S^2}\\
&\qquad=\lim_{\varepsilon\rightarrow 0}\lim_{\nu\rightarrow\infty}
\int_{B_{\varepsilon}}\abs{\psi^{\nu}}^4\dvol_{S^2}
\end{align*}
By Step 4 in the proof of Prp. 4.7.1 in \cite{MS04}, all bubbling points $z_1,\ldots,z_l$ of the sequence
$(\tvarphi^{\nu},\tpsi^{\nu})$ may be assumed to lie in $\overline{B_1}$.
We fix a number $s>1$. Then, by the identity
just established and conformal invariance (\ref{eqnSuperEnergyInvariance}) of the super energy,
the following calculation is valid.
\begin{align*}
m_0^{\psi}&=\lim_{R\rightarrow\infty}\lim_{\nu\rightarrow\infty}E(\psi^{\nu},B_{R\delta^{\nu}})\\
&=\lim_{R\rightarrow\infty}\lim_{\nu\rightarrow\infty}E(\tpsi^{\nu},B_{R})\\
&=\lim_{R\rightarrow\infty}\lim_{\nu\rightarrow\infty}E(\tpsi^{\nu},B_{R}\setminus B_s)
+\lim_{\nu\rightarrow\infty}E(\tpsi^{\nu},B_s)\\
&=\lim_{R\rightarrow\infty}E(\tpsi,B_{R}\setminus B_s)
+\lim_{\nu\rightarrow\infty}E(\tpsi^{\nu},B_s)\\
&=E(\tpsi,S^2\setminus B_s)+\lim_{\varepsilon\rightarrow 0}\lim_{\nu\rightarrow\infty}
E\left(\tpsi^{\nu},B_s\setminus\bigcup_{j=1}^lB_{\varepsilon}(z_j)\right)+\sum_{j=1}^lm_j^{\psi}\\
&=E(\tpsi,S^2)+\sum_{j=1}^lm_j^{\psi}
\end{align*}
This completes the proof.
\end{proof}

\section{Gromov Compactness}
\label{secGromovCompactness}

In this section, we prove Gromov compactness for sequences of holomorphic supercurves
with domain $S^2$ and holomorphic line bundle $L=L_{-1}=S^+$, provided that the
super energy is uniformly bounded. In other words, we show that every such sequence has
a subsequence that converges in a sense to be made precise.
As in the previous section, let $(X,\omega)$ be a compact symplectic manifold and
let $\mA:=\mA(GL(X))$ and $\mJ:=\mJ(X,\omega)$ denote the respective spaces of connections
and $\omega$-tame almost complex structures on $X$. Moreover, we fix $(A,J)\in\mA\times\mJ$.

By the results of Sec. \ref{secBubbling}, convergence cannot be understood without taking
into account the bubbling off of holomorphic superspheres. This results in a tree of holomorphic
supercurves with compatibility conditions for the edges as the limiting object.
The precise notion is that of a stable (holomorphic) supercurve, to be introduced first.
Moreover, we show Gromov compactness for sequences of such stable supercurves and finally
introduce a compact and metrisable topology on the moduli spaces of equivalence classes
with respect to which convergence is equivalent to Gromov convergence.
Throughout, we prove the results for supercurves endowed with marked points,
which might be useful for defining geometric invariants in subsequent work.

The technical difficulties occurring here have, for the most part, already been solved in the
previous sections, most notably by the removable singularity Theorem \ref{thmRemovalOfSingularities}
and the super energy identity in Prp. \ref{prpSuperBubblesEnergy}. With these issues
settled, the following treatment is, to a large extent, parallel to classical Gromov
compactness for holomorphic curves as covered e.g. in Chp. 5 of \cite{MS04}.
We will thus often refer the reader to the arguments provided there and prove
only the sectional parts of our statements, unless it is more instructive to
establish the results a whole.
While the full theory is obtained only for $L=L_{-1}=S^+$, we yield
partial results also for the case $L=L_{-2}=K$.

We recall some preliminaries on trees first, consult App. D in \cite{MS04} for a thorough treatment.
A \emph{tree} is a connected graph without cycles, and as such consists of a finite set $T$
and a relation $E\subseteq T\times T$ such that any two \emph{vertices} (= elements) $\alpha,\beta\in T$
are connected by an edge if and only if $\alpha E\beta$.
The vertices $\alpha_1,\ldots,\alpha_N$ of any tree $(T,E)$ can be ordered such that
the restriction to the subset $T_i:=\{\alpha_1,\ldots,\alpha_i\}$ is a tree for every $i$. Moreover,
for every $i\geq 2$ there is a unique index $j_i<i$ such that $\alpha_{j_i}E\alpha_i$.
A tree $(T,E)$ with $N$ vertices can, therefore, be identified with an integer vector
$(j_2,\ldots,j_N)$ such that $1\leq j_i<i$ for every $i$.
Removing the edge connecting $\alpha E\beta$ results in two subtrees. We denote the
one containing $\beta$ by
\begin{align*}
T_{\alpha\beta}:=\{\gamma\in T\setsep\beta\in[\alpha,\gamma]\}\setminus\{\alpha\}
\end{align*}
where $[\alpha,\gamma]$ is, by definition, the set of elements of $T$ along the chain of
edges connecting $\alpha$ and $\beta$.
Finally, we need the following notion.
An \emph{$n$-labelling} $\Lambda$ of a tree $(T,E)$ is a function
$\Lambda:\{1,\ldots,n\}\rightarrow T\,,\; i\mapsto\alpha_i$
that attaches labels to vertices of $T$.

\begin{Def}[Stable Supercurve]
\label{defStableSupercurve}
Let $n\geq 0$ be a non-negative integer.
A \emph{stable $(A,J)$-holomorphic supercurve from $(S^2,L_d)$ to $X$ with $n$ marked points,
modelled over the labelled tree $(T,E,\Lambda)$} is a tuple
\begin{align*}
((\vvarphi,\vpsi),\vz)=\left(\{\varphi_{\alpha},\psi_{\alpha}\}_{\alpha\in T}
,\{z_{\alpha\beta}\}_{\alpha E\beta},\{\alpha_i,z_i\}_{1\leq i\leq n}\right)
\end{align*}
consisting of a collection of $(A,J)$-holomorphic supercurves
$(\varphi_{\alpha},\psi_{\alpha}):(S^2,L_d)\rightarrow X$ labelled by vertices $\alpha\in T$,
a collection of \emph{nodal points} $z_{\alpha\beta}\in S^2$ labelled by the oriented
edges $\alpha E\beta$, and a sequence of $n$ \emph{marked points} $z_1,\ldots,z_n\in S^2$,
such that the following conditions are satisfied.
\begin{itemize}
\item\emph{Nodal Points:} $\alpha E\beta$ implies $\varphi_{\alpha}(z_{\alpha\beta})=\varphi_{\beta}(z_{\beta\alpha})$
We denote the sets of nodal points on the $\alpha$-sphere by $Z_{\alpha}:=\{z_{\alpha\beta}\setsep\alpha E\beta\}$.
\item\emph{Special Points:} The points $z_{\alpha\beta}$ (for $\alpha E\beta$) and
$z_i$ (such that $\alpha_i=\alpha$)
are pairwise distinct. We denote the set of special points on the $\alpha$-sphere by
$Y_{\alpha}:=Z_{\alpha}\cup\{z_i\setsep\alpha_i=\alpha\}$.
\item\emph{Stability:} If $\varphi_{\alpha}$ is a constant function, then $\#Y_{\alpha}\geq 3$
and $\psi_{\alpha}\equiv 0$.
\end{itemize}
\end{Def}

Note that, by Lem. \ref{lemGhostBubble}, the second part of the stability condition is automatically satisfied
in the cases $d=\deg L_d<0$.

For a stable supercurve $((\vvarphi,\vpsi),\vz)$, we denote
\begin{align*}
E(\vvarphi):=\sum_{\alpha\in T}E(\varphi_{\alpha})\;&,\qquad
E(\vpsi):=\sum_{\alpha\in T}E(\psi_{\alpha})\\
m_{\alpha\beta}(\vvarphi):=\sum_{\gamma\in T_{\alpha\beta}}E(\varphi_{\gamma})\;&,\qquad
m_{\alpha\beta}(\vpsi):=\sum_{\gamma\in T_{\alpha\beta}}E(\psi_{\gamma})
\end{align*}

\begin{Def}
\label{defStableSupercurveEquivalence}
Two stable supercurves $((\vvarphi,\vpsi),\vz)$ and $((\tilde{\vvarphi},\tilde{\vpsi}),\tilde{\vz})$,
modelled over the labelled trees $(T,\Lambda)$ and $(\tilde{T},\tilde{\Lambda})$, respectively,
are called \emph{equivalent} if there exists a tree isomorphism $f:T\rightarrow\tilde{T}$ and a
function $T\rightarrow G,\,\alpha\mapsto m_{\alpha}$ that assigns to each vertex of $T$
a Möbius transformation, such that
\begin{align*}
\tvarphi_{f(\alpha)}\circ m_{\alpha}=\varphi_{\alpha}\;&,\qquad
\tpsi_{f(\alpha)}\circ m_{\alpha}=\psi_{\alpha}\\
\tz_{f(\alpha)f(\beta)}=m_{\alpha}(z_{\alpha\beta})\;&,\qquad\tz_i=m_{\alpha_i}(z_i)
\end{align*}
\end{Def}

\begin{Def}
\label{defStableSuperModuliSpaces}
Let $E>0$ and $\beta\in H_2(X,\bZ)$ be a homology class. We define the \emph{moduli spaces}
\begin{itemize}
\item $\hat{\mM}_{0,d,n}(X;A,J)$ of stable $(A,J)$-holomorphic
supercurves $((\vvarphi,\vpsi),\vz)$ from $(S^2,L_d)$ to $X$ with $n$ marked points,
\item $\hat{\mM}_{0,d,n}(X,\beta,E;A,J)$ of those which represent
$\beta$ in the sense $\beta=\sum_{\alpha\in T}{\varphi_{\alpha}}_*\left[S^2\right]$
and have bounded super energy
$E(\vpsi)\leq E$,
\item $\overline{\mM}_{0,d,n}(X,\beta,E;A,J)$ of equivalence classes, in the sense
of Def. \ref{defStableSupercurveEquivalence}, of stable
supercurves in $\hat{\mM}_{0,d,n}(X,\beta,E;A,J)$.
\end{itemize}
\end{Def}

\begin{Rem}
Since equivalence of stable supercurves is defined by means of Möbius transformations,
the property of representing $\beta$ is obviously invariant and,
by conformal invariance (\ref{eqnSuperEnergyInvariance}) of the super energy,
the same is true for the condition $E(\vpsi)\leq E$.
The moduli space $\overline{\mM}_{0,d,n}(X,\beta,E;A,J)$ of equivalence
classes is hence indeed well-defined.

For the classical energy, it is well-known that there is a constant $E'>0$
(depending on $\beta$) such that $E(\vvarphi)\leq E'$ holds
whenever $\vvarphi$ represents $\beta$.
Uniform boundedness of the super energy $E(\vvarphi)+E(\vpsi)$ will turn out to be
crucial for establishing Gromov compactness (cf. Thm. \ref{thmSuperGromovCompactness} and Thm.
\ref{thmSuperGromovCompactnessStable} below).
\end{Rem}

For the next definition and later use, we introduce the following abbreviation
that we could have used already in the previous section.
Let $U\subseteq S^2$ be an open subset.
A sequence $(\varphi^{\nu},\psi^{\nu})$ is said to \emph{converge u.c.s. on $U$}
to a pair $(\varphi,\psi)$ if it converges to $(\varphi,\psi)$ in the $C^{\infty}$-topology on
every compact subset of $U$
(if it converges \emph{\textbf{u}}niformly with all derivatives on
\emph{\textbf{c}}ompact \emph{\textbf{s}}ubsets).

\begin{Def}[Gromov Convergence]
\label{defSuperGromovConvergence}
Let $(A^{\nu},J^{\nu})\in\mA\times\mJ$ be a sequence that converges to $(A,J)$ in the
$C^{\infty}$-topology, let
\begin{align*}
((\vvarphi,\vpsi),\vz)=\left(\{\varphi_{\alpha},\psi_{\alpha}\}_{\alpha\in T}
,\{z_{\alpha\beta}\}_{\alpha E\beta},\{\alpha_i,z_i\}_{1\leq i\leq n}\right)
\end{align*}
be a stable supercurve and let $(\varphi^{\nu},\psi^{\nu}):(S^2,L_d)\rightarrow X$ be a sequence of
$(A^{\nu},J^{\nu})$-holomorphic supercurves with $n$ distinct marked points $z_1^{\nu},\ldots,z_n^{\nu}\in S^2$.
The sequence
\begin{align*}
((\varphi^{\nu},\psi^{\nu}),\vz^{\nu}):=((\varphi^{\nu},\psi^{\nu}),z_1^{\nu},\ldots,z_n^{\nu})
\end{align*}
is said to \emph{Gromov converge} to $((\vvarphi,\vpsi),\vz)$ if there exists a collection of Möbius
transformations $\{m_{\alpha}^{\nu}\}_{\alpha\in T}^{\nu\in\bN}$ such that the following
axioms are satisfied.
\begin{itemize}
\item\emph{Map:} For every $\alpha\in T$, the sequence
\begin{align*}
(\varphi^{\nu}_{\alpha},\psi^{\nu}_{\alpha})
:=(\varphi^{\nu}\circ m_{\alpha}^{\nu},\psi^{\nu}\circ m_{\alpha}^{\nu}):(S^2,L_d)\rightarrow X
\end{align*}
converges to $(\varphi_{\alpha},\psi_{\alpha})$ u.c.s. on $S^2\setminus Z_{\alpha}$.
\item\emph{Energy:} If $\alpha E\beta$ then
\begin{align*}
m_{\alpha\beta}(\vvarphi)
&=\lim_{\varepsilon\rightarrow 0}\lim_{\nu\rightarrow\infty} E(\varphi_{\alpha}^{\nu},B_{\varepsilon}(z_{\alpha\beta}))\\
m_{\alpha\beta}(\vpsi)
&=\lim_{\varepsilon\rightarrow 0}\lim_{\nu\rightarrow\infty} E(\psi_{\alpha}^{\nu},B_{\varepsilon}(z_{\alpha\beta}))
\end{align*}
\item\emph{Rescaling:} If $\alpha E\beta$ then the sequence
$m_{\alpha\beta}^{\nu}:=(m_{\alpha}^{\nu})^{-1}\circ m_{\beta}^{\nu}$ converges to $z_{\alpha\beta}$
u.c.s. on $S^2\setminus\{z_{\beta\alpha}\}$.
\item\emph{Marked Points:} $z_i=\lim_{\nu\rightarrow\infty}(m_{\alpha_i}^{\nu})^{-1}(z_i^{\nu})$ for $i=1,\ldots,n$.
\end{itemize}
\end{Def}

Setting $\psi^{\nu}:=0$ and $\vpsi:=0$ in Def. \ref{defSuperGromovConvergence} recovers
Gromov convergence for the underlying sequence of holomorphic curves. In particular,
this implies the following proposition, proved as Thm. 5.2.2(ii) in \cite{MS04}.
Connected sums are explained in \cite{Hat01}.

\begin{Prp}
\label{prpSuperGromovConvergence}
Let $((\varphi^{\nu},\psi^{\nu}),\vz^{\nu})$ be a sequence of holomorphic supercurves
with $n$ marked points that Gromov converges
to the stable supercurve $((\vvarphi,\vpsi),\vz)$ as in Def. \ref{defSuperGromovConvergence}.
Then, for sufficiently large $\nu$, the map $\varphi^{\nu}:S^2\rightarrow X$ is homotopic to the connected
sum $\#_{\alpha\in T}\varphi_{\alpha}$.
\end{Prp}

\begin{Thm}[Gromov Compactness]
\label{thmSuperGromovCompactness}
Let $(A^{\nu},J^{\nu})\in\mA\times\mJ$ be a sequence that converges to $(A,J)$ in the
$C^{\infty}$-topology.
Let $(\varphi^{\nu},\psi^{\nu}):(S^2,L_d)\rightarrow X$ be a sequence of $(A^{\nu},J^{\nu})$-holomorphic
supercurves such that the super energy
\begin{align*}
\sup_{\nu}E(\varphi^{\nu},\psi^{\nu})<\infty
\end{align*}
is uniformly bounded, and
$\vz^{\nu}=(z_1^{\nu},\ldots,z_n^{\nu})$ be a sequence of $n$-tuples of pairwise distinct points in $S^2$.
Then, in the case $d=-1$, $(\varphi^{\nu},\psi^{\nu},\vz^{\nu})$
has a Gromov convergent subsequence.
\end{Thm}

\begin{proof}[Proof for $n=0$]
Since, by assumption, the super energy is uniformly bounded, we may assume that
$E(\varphi^{\nu})$ and $E(\psi^{\nu})$ converge. We denote the limits by
\begin{align*}
E^{\varphi}:=\lim_{\nu}E(\varphi^{\nu})\;,\qquad E^{\psi}:=\lim_{\nu}E(\psi^{\nu})
\end{align*}
As indicated in the beginning of this section, we describe a tree with $N$ vertices
by an integer vector $(j_2,\ldots,j_N)$. Following the proof of Thm. 5.3.1 in \cite{MS04}
we shall, by induction, construct
\begin{itemize}
\item a tuple $\vx:=((\varphi_1,\psi_1),\ldots,(\varphi_N,\psi_N);j_2,\ldots j_N;z_2,\ldots z_N)$
that consists of $(A,J)$-holomorphic supercurves $(\varphi_i,\psi_i):(S^2,L_d)\rightarrow X$,
positive integers $j_i<1$
for $i\geq 2$, and complex numbers $z_i\in\bC$ with $\abs{z_i}\leq 1$,
\item finite subsets $Z_i\subseteq B_1$ for $i=1,\ldots,N$,
\item sequences of Möbius transformations $\{m_i^{\nu}\}_{\nu\in\bN}$ for $i=1,\ldots,N$
\end{itemize}
such that a suitable subsequence satisfies the following conditions.
\begin{enumerate}
\renewcommand{\labelenumi}{(\roman{enumi})}
\item $(\varphi^{\nu}\circ m_1^{\nu},\psi^{\nu}\circ m_1^{\nu})$ converges
to $(\varphi_1,\psi_1)$ u.c.s on $S^2\setminus Z_1$. For $i=2,\ldots,N$,
$(\varphi^{\nu}\circ m_i^{\nu},\psi^{\nu}\circ m_i^{\nu})$ converges
to $(\varphi_i,\psi_i)$ u.c.s. on $\bC\setminus Z_i$. Moreover, $Z_1\subseteq\{0\}$ and $Z_N=\emptyset$.
\item If $\varphi_i$ is constant then $\#Z_i\geq 2$ and $\psi_i\equiv 0$.
\item The limits
\begin{align*}
m_i^{\varphi}(z)&:=\lim_{\varepsilon\rightarrow 0}\lim_{\nu\rightarrow\infty}
E(\varphi^{\nu}\circ m_i^{\nu},B_{\varepsilon}(z))\\
m_i^{\psi}(z)&:=\lim_{\varepsilon\rightarrow 0}\lim_{\nu\rightarrow\infty}
E(\psi^{\nu}\circ m_i^{\nu},B_{\varepsilon}(z))
\end{align*}
exist for all $z\in Z_i$, and are such that $m_i^{\varphi}>0$, $m_i^{\psi}\geq 0$ holds. Moreover,
if $Z_1=\{0\}$ then
\begin{align*}
E^{\varphi}=E(\varphi_1)+m_1^{\varphi}(0)\;,\qquad E^{\psi}=E(\psi_1)+m_1^{\psi}(0)
\end{align*}
If $i\geq 2$ then $z_i\in Z_{j_i}$ and
\begin{align*}
m_{j_i}^{\varphi}(z_i)=E(\varphi_i)+\sum_{z\in Z_i}m_i^{\varphi}(z)\;,\qquad
m_{j_i}^{\psi}(z_i)=E(\psi_i)+\sum_{z\in Z_i}m_i^{\psi}(z)
\end{align*}
\item If $j_k=j_{k'}$ then $z_k\neq z_{k'}$.
\item For $i=2,\ldots,N$, $\varphi_{j_i}(z_i)=\varphi_i(\infty)$.
\item For $i=2,\ldots N$, $(m_{j_i}^{\nu})^{-1}\circ m_i^{\nu}$ converges to $z_i$ u.c.s.
on $\bC=S^2\setminus\{\infty\}$.
\item For $i=1,\ldots N$, $Z_i=\{z_k\setsep i<k\leq N, j_k=i\}$.
\end{enumerate}
The proof starts by constructing $(\varphi_1,\psi_1)$ to have at most one bubbling point
and then proceeds by induction, constructing the $(\varphi_i,\psi_i)$ and $Z_i$ so as
to satisfy (i)-(vi). When the induction is complete (vii) will also be satisfied.
This is a straightforward generalisation of the proof of Thm. 5.3.1 in \cite{MS04},
which works as follows.

In the base case, $\varphi^{\nu}$ is rescaled to a sequence
of which a subsequence $\varphi'^{\nu}$ converges u.c.s. on $\bC=S^2\setminus\{\infty\}$.
The construction of the rescaled sequence $\psi'^{\nu}$ is analogous, and u.c.s. convergence
of a (further) subsequence of $(\varphi'^{\nu},\psi'^{\nu})$ follows with
Prp. \ref{prpPsiLInfinityBounded}.

For the inductive step, let $l\geq 1$ and suppose, by induction, that $(\varphi_i,\psi_i)$, $j_i$, $z_i$, $Z_i$ and
$\{m_i^{\nu}\}_{\nu\in\bN}$ have been constructed for $i\leq l$ so as to satisfy (i)-(vi)
but not (vii), with $N$ replaced by $l$. We define
$Z_{j;l}:=Z_j\setminus\{z_i\setsep j<i\leq l\,,\;j_i=j\}$
for $j\leq l$. Since (vii) is not satisfied, there exists $j$
such that $Z_{j;l}$ is non-empty. Let $z_{l+1}\in Z_{j;l}$ and apply
Prp. \ref{prpSuperBubbles} and Prp. \ref{prpSuperBubblesEnergy} to the sequence
$(\varphi^{\nu}\circ m_j^{\nu},\psi^{\nu}\circ m_j^{\nu})$ and the point $z_0:=z_{l+1}$
to yield a subsequence, also denoted $\nu$, as well as
Möbius transformations $m^{\nu}$ and a finite set $Z\subseteq B_1$ such that
(i)-(iv) as well as (vi) follow from the conclusions of those propositions.
(v) is established using Prp. \ref{prpBubblesConnect}.

It remains to verify that the induction terminates after finitely many steps,
$\vx$ as constructed is indeed a stable supercurve, and
$(\varphi^{\nu},\psi^{\nu})$ Gromov converges to $\vx$.
This works as in the proof of the classical Gromov compactness theorem.
\end{proof}

For the proof of Thm. \ref{thmSuperGromovCompactness} for $n>0$ marked points, we
need the following two lemmas.

\begin{Lem}
\label{lemSuperNoFurtherBubbling}
Let $(A^{\nu},J^{\nu})\in\mA\times\mJ$ be a sequence that converges to $(A,J)$ in the
$C^{\infty}$-topology, and suppose that the sequence
$((\varphi^{\nu},\psi^{\nu}),z_1^{\nu},\ldots,z_n^{\nu})$ of $(A^{\nu},J^{\nu})$-holomorphic supercurves
with marked points Gromov converges to the stable supercurve
$((\vvarphi,\vpsi),\vz)\in\hat{\mM}_{0,d,n}(X;A,J)$
via the reparametrisation sequences $m_{\alpha}^{\nu}\in G$. Moreover,
let $\alpha E\beta$ be an edge and $m^{\nu}\in G$ be a sequence such that
\begin{enumerate}
\renewcommand{\labelenumi}{(\alph{enumi})}
\item $(m_{\alpha}^{\nu})^{-1}\circ m^{\nu}$ converges to $z_{\alpha\beta}$ u.c.s. on $S^2\setminus\{w_0\}$.
\item $(m_{\beta}^{\nu})^{-1}\circ m^{\nu}$ converges to $z_{\beta\alpha}$ u.c.s. on $S^2\setminus\{w_1\}$.
\end{enumerate}
Then $\varphi^{\nu}\circ m^{\nu}$ converges to $\varphi_{\alpha}(z_{\alpha\beta})=\varphi_{\beta}(z_{\beta\alpha})$
u.c.s. on $S^2\setminus\{w_0,w_1\}$ and
\begin{align*}
\lim_{\nu\rightarrow\infty}E(\varphi^{\nu}\circ m^{\nu},B_r(w_0))=m_{\beta\alpha}(\vvarphi)\,&,\quad
\lim_{\nu\rightarrow\infty}E(\psi^{\nu}\circ m^{\nu},B_r(w_0))=m_{\beta\alpha}(\vpsi)\\
\lim_{\nu\rightarrow\infty}E(\varphi^{\nu}\circ m^{\nu},B_r(w_1))=m_{\alpha\beta}(\vvarphi)\,&,\quad
\lim_{\nu\rightarrow\infty}E(\psi^{\nu}\circ m^{\nu},B_r(w_1))=m_{\alpha\beta}(\vpsi)
\end{align*}
whenever $r<d_{S^2}(w_0,w_1)$. Moreover, if $\alpha_i\in T_{\alpha\beta}$ then $(m^{\nu})^{-1}(z_i^{\nu})$
converges to $w_1$, and if $\alpha_i\in T_{\beta\alpha}$ then $(m^{\nu})^{-1}(z_i^{\nu})$
converges to $w_0$.
\end{Lem}

\begin{proof}
The statements concerning $\varphi^{\nu}$ and $z_i^{\nu}$ are shown in Lem. 5.4.2 of \cite{MS04},
and thus it remains to establish the energy identities concerning $\vpsi$.
Using conformal invariance (\ref{eqnSuperEnergyInvariance}) of the super energy,
the proof is analogous and thus omitted here.
\end{proof}

The following lemma, which is proved as Lem. 5.3.3 in \cite{MS04}, concerns only the
underlying sequence of holomorphic curves with marked points.
We formulate it in terms of holomorphic supercurves.

\begin{Lem}
\label{lemMarkedPointsCases}
Suppose that $((\varphi^{\nu},0),z_1^{\nu},\ldots,z_k^{\nu})$ Gromov converges to a stable supercurve
$((\vvarphi,\vec{0}),\vz)\in\hat{\mM}_{0,d,k}(X;A,J)$
modelled over the tree $T$ via $m_{\alpha}^{\nu}\in G$. Let
$\zeta^{\nu}\in S^2\setminus\{z_1^{\nu},\ldots,z_k^{\nu}\}$ be a sequence such that
the limits $\zeta_{\alpha}:=\lim_{\nu\rightarrow\infty}(m_{\alpha}^{\nu})^{-1}(\zeta^{\nu})$ 
exist for all $\alpha\in T$. Then precisely one of the following conditions holds.
\begin{itemize}
\item[(I)] There exists a (unique) vertex $\alpha\in T$ such that
\begin{align*}
\zeta_{\alpha}\notin Y_{\alpha}((\vvarphi,\vec{0}),\vz)=Z_{\alpha}\cup\{z_i\setsep\alpha_i=\alpha\}
\end{align*}
\item[(II)] There exists a (unique) index $i\in\{1,\ldots,k\}$ such that $\zeta_{\alpha_i}=z_i$.
\item[(III)] There exists a (unique) edge $\alpha E\beta$ in $T$ such that $\zeta_{\alpha}=z_{\alpha\beta}$
and $\zeta_{\beta}=z_{\beta\alpha}$.
\end{itemize}
\end{Lem}

\begin{proof}[Proof of Thm. \ref{thmSuperGromovCompactness} for $n>0$]
Let $((\varphi^{\nu},\psi^{\nu}),z_1^{\nu},\ldots,z_n^{\nu})$ denote a sequence of $(A^{\nu},J^{\nu})$-holomorphic
supercurves, each with $n$ distinct marked points. The strategy is to prove, by induction over $k$, that a subsequence
of $((\varphi^{\nu},\psi^{\nu}),z_1^{\nu},\ldots,z_k^{\nu})$ Gromov converges to a stable supercurve
\begin{align*}
((\vvarphi,\vpsi),\vz)=\left(\{\varphi_{\alpha},\psi_{\alpha}\}_{\alpha\in T}
,\{z_{\alpha\beta}\}_{\alpha E\beta},\{\alpha_i,z_i\}_{1\leq i\leq k}\right)
\in\hat{\mM}_{0,d,k}(X;A,J)
\end{align*}
via Möbius transformations $m_{\alpha}^{\nu}$.
We have already proved that this holds for $k=0$.
Let $k\geq 1$ and assume, by induction, that the statement has been established for $k-1$.
Passing to a further subsequence, if necessary, we may assume that the limits
$z_{\alpha k}:=\lim_{\nu}(m_{\alpha}^{\nu})^{-1}(z_k^{\nu})$
exist for all vertices $\alpha$. Hence, we may apply Lem. \ref{lemMarkedPointsCases},
by which precisely one of the conditions I, II or III holds.
In each case, the proof of Thm. 5.3.1 in \cite{MS04} shows that the underlying sequence $(\varphi^{\nu},z_1^{\nu},\ldots,z_k^{\nu})$ of $J^{\nu}$-holomorphic curves with $k$ marked points has a Gromov convergent subsequence.
The remaining axioms of Gromov convergence as in Def. \ref{defSuperGromovConvergence},
concerning the sequence $\psi^{\nu}$, are then shown similarly,
using Prp. \ref{prpPsiLInfinityBounded} and Lem. \ref{lemGhostBubble}.
\end{proof}

\begin{Thm}[Uniqueness of the Limit]
\label{thmSuperGromovCompactnessUniqueness}
Let $(A^{\nu},J^{\nu})\in\mA\times\mJ$ be a sequence that converges to $(A,J)$ in the
$C^{\infty}$-topology, and let $((\varphi^{\nu},\psi^{\nu}),z_1^{\nu},\ldots,z_n^{\nu})$
be a sequence of $(A^{\nu},J^{\nu})$-holomorphic
supercurves with $n$ distinct marked points that Gromov converges to two stable supercurves
$((\vvarphi,\vpsi),\vz)$ and $((\tilde{\vvarphi},\tilde{\vpsi}),\tilde{\vz})$.
Then $((\vvarphi,\vpsi),\vz)$ and $((\tilde{\vvarphi},\tilde{\vpsi}),\tilde{\vz})$ are equivalent
(in the sense of Def. \ref{defStableSupercurveEquivalence}).
\end{Thm}

\begin{proof}
The hypotheses imply, in particular, that the underlying sequence $(\varphi^{\nu},z_1^{\nu},\ldots,z_n^{\nu})$
of $J^{\nu}$-holomorphic spheres with marked points Gromov converges to the two stable maps
$(\vvarphi,\vz)$ and $(\tilde{\vvarphi},\tilde{\vz})$.
By Thm. 5.4.1. in \cite{MS04}, both are equivalent, 
i.e. there is a tree isomorphism $f:T\rightarrow\tilde{T}$ and a collection
of Möbius transformations $\{m_{\alpha}\}_{\alpha\in T}$ such that
\begin{align*}
\tvarphi_{f(\alpha)}=\varphi_{\alpha}\circ m_{\alpha}\;,\qquad
\tz_{f(\alpha)f(\beta)}=m^{-1}_{\alpha}(z_{\alpha\beta})\;,\qquad\tz_i=m^{-1}_{\alpha_i}(z_i)
\end{align*}
By the proof of that theorem,
$m_{\alpha}=\lim_{\nu\rightarrow\infty}(m_{\alpha}^{\nu})^{-1}\circ\tilde{m}^{\nu}_{f(\alpha)}$ holds,
where $\{m_{\alpha}^{\nu}\}_{\alpha\in T}$ and $\{\tilde{m}_{\alpha}^{\nu}\}_{\alpha\in\tilde{T}}$
denote the respective sequences of Möbius transformations from Definition \ref{defSuperGromovConvergence}
of Gromov convergence. In particular, the limit on the right hand side exists.
We thus obtain the following identity.
\begin{align*}
\psi_{\alpha}\circ m_{\alpha}
&=\lim_{\nu\rightarrow\infty}\psi_{\alpha}\circ(m_{\alpha}^{\nu})^{-1}\circ\tilde{m}^{\nu}_{f(\alpha)}
=\lim_{\nu\rightarrow\infty}\left(\psi^{\nu}\circ m_{\alpha}^{\nu}\right)
\circ(m_{\alpha}^{\nu})^{-1}\circ\tilde{m}^{\nu}_{f(\alpha)}\\
&=\lim_{\nu\rightarrow\infty}\psi^{\nu}\circ\tilde{m}^{\nu}_{f(\alpha)}
=\tpsi_{f(\alpha)}
\end{align*}
where the limits are understood modulo bubbling. Therefore, the stable supercurves
$((\vvarphi,\vpsi),\vz)$ and $((\tilde{\vvarphi},\tilde{\vpsi}),\tilde{\vz})$ are equivalent
as a whole, which was to show.
\end{proof}

\subsection{Compactness for Stable Supercurves}

Based on the results for sequences of holomorphic supercurves already achieved, we prove
in this subsection Gromov compactness and uniqueness of limits, up to equivalence,
for sequences of stable supercurves.
For a stable supercurve $((\vvarphi,\vpsi),\vz)$ modelled over a tree $T$, a vertex $\alpha\in T$
and an open set $U_{\alpha}\in S^2$, we denote
\begin{align*}
E_{\alpha}(\vvarphi,U_{\alpha})&:=E(\varphi_{\alpha},U_{\alpha})
+\sum_{\beta\in T,\,\alpha E\beta,\, z_{\alpha\beta}\in U_{\alpha}}m_{\alpha\beta}(\vvarphi)\\
E_{\alpha}(\vpsi,U_{\alpha})&:=E(\psi_{\alpha},U_{\alpha})
+\sum_{\beta\in T,\,\alpha E\beta,\, z_{\alpha\beta}\in U_{\alpha}}m_{\alpha\beta}(\vpsi)
\end{align*}

\begin{Def}[Gromov Convergence]
\label{defSuperGromovConvergenceStable}
Let $(A^{\nu},J^{\nu})\in\mA\times\mJ$ be a sequence that converges to $(A,J)$ in the
$C^{\infty}$-topology. A sequence of stable supercurves
\begin{align*}
((\vvarphi^{\nu},\vpsi^{\nu}),\vz^{\nu})=\left(\{\varphi_{\alpha}^{\nu},\psi_{\alpha}^{\nu}\}_{\alpha\in T^{\nu}}
,\{z^{\nu}_{\alpha\beta}\}_{\alpha E^{\nu}\beta},\{\alpha_i^{\nu},z_i^{\nu}\}_{1\leq i\leq n}\right)
\in\hat{\mM}_{0,d,n}(X;A^{\nu},J^{\nu})
\end{align*}
is said to \emph{Gromov converge} to a stable supercurve
\begin{align*}
((\vvarphi,\vpsi),\vz)=\left(\{\varphi_{\alpha},\psi_{\alpha}\}_{\alpha\in T}
,\{z_{\alpha\beta}\}_{\alpha E\beta},\{\alpha_i,z_i\}_{1\leq i\leq n}\right)
\in\hat{\mM}_{0,d,n}(X;A,J)
\end{align*}
if, for every sufficiently large $\nu$, there exists a surjective tree homomorphism $f^{\nu}:T\rightarrow T^{\nu}$
and a collection of Möbius transformations $\{m_{\alpha}^{\nu}\}_{\alpha\in T}^{\nu\in\bN}$ such that the
following axioms are satisfied.
\begin{itemize}
\item\emph{Map:} For every $\alpha\in T$ the sequence
\begin{align*}
(\varphi^{\nu}_{f^{\nu}(\alpha)}\circ m_{\alpha}^{\nu},
\psi^{\nu}_{f^{\nu}(\alpha)}\circ m_{\alpha}^{\nu}):(S^2,L_d)\rightarrow X
\end{align*}
converges to $(\varphi_{\alpha},\psi_{\alpha})$ u.c.s. on $S^2\setminus Z_{\alpha}$.
\item\emph{Energy:} If $\alpha E\beta$ then
\begin{align*}
m_{\alpha\beta}(\vvarphi)&=\lim_{\varepsilon\rightarrow 0}\lim_{\nu\rightarrow\infty}
E_{f^{\nu}(\alpha)}(\vvarphi^{\nu},m_{\alpha}^{\nu}(B_{\varepsilon}(z_{\alpha\beta})))\\
m_{\alpha\beta}(\vpsi)&=\lim_{\varepsilon\rightarrow 0}\lim_{\nu\rightarrow\infty}
E_{f^{\nu}(\alpha)}(\vpsi^{\nu},m_{\alpha}^{\nu}(B_{\varepsilon}(z_{\alpha\beta})))
\end{align*}
\item\emph{Rescaling:} If $\alpha,\beta\in T$ such that $\alpha E\beta$ and $\nu_j$ is a subsequence such that
$f^{\nu_j}(\alpha)=f^{\nu_j}(\beta)$ then $m_{\alpha\beta}^{\nu_j}:=(m_{\alpha}^{\nu_j})^{-1}\circ m_{\beta}^{\nu_j}$
converges to $z_{\alpha\beta}$ u.c.s. on $S^2\setminus\{z_{\beta\alpha}\}$.
\item\emph{Nodal Points:} If $\alpha,\beta\in T$ such that $\alpha E\beta$ and $\nu_j$ is a subsequence such that
$f^{\nu_j}(\alpha)\neq f^{\nu_j}(\beta)$ then
$z_{\alpha\beta}=\lim_{j\rightarrow\infty}(m_{\alpha}^{\nu_j})^{-1}(z^{\nu_j}_{f^{\nu_j}(\alpha)f^{\nu_j}(\beta)})$.
\item\emph{Marked Points:} $\alpha_i^{\nu}=f^{\nu}(\alpha_i)$ and
$z_i=\lim_{\nu\rightarrow\infty}(m_{\alpha_i}^{\nu})^{-1}(z_i^{\nu})$ for all $i=1,\ldots,n$.
\end{itemize}
\end{Def}

\begin{Thm}[Gromov Compactness]
\label{thmSuperGromovCompactnessStable}
Let $(A^{\nu},J^{\nu})\in\mA\times\mJ$ be a sequence that converges to $(A,J)$ in the
$C^{\infty}$-topology. Let
$((\vvarphi^{\nu},\vpsi^{\nu}),\vz^{\nu})\in\hat{\mM}_{0,d,n}(X;A^{\nu},J^{\nu})$ be a sequence
of stable supercurves such that
\begin{align*}
\sup_{\nu}E(\vvarphi^{\nu})<\infty\;,\qquad\sup_{\nu}E(\vpsi^{\nu})<\infty
\end{align*}
Then, in the case $d=-1$, $((\vvarphi^{\nu},\vpsi^{\nu}),\vz^{\nu})$
has a Gromov convergent subsequence.
\end{Thm}

\begin{Rem}
\label{remSuperGromovCompactnessStable}
The proofs of Thm. \ref{thmSuperGromovCompactnessStable} and the subsequent results
simplify due to the following general principle. By the proof of Thm. 5.5.5 in \cite{MS04}
there are, up to isomorphism, only finitely many $n$-labelled trees which underlie stable
$J^{\nu}$-holomorphic curves with energy bounded by a constant, provided that the
sequence $J^{\nu}$ is convergent. This fact is shown by properties of weighted trees.
\end{Rem}

\begin{proof}[Proof of Thm. \ref{thmSuperGromovCompactnessStable}]
Passing to a subsequence we may assume, by Rem. \ref{remSuperGromovCompactnessStable},
that all trees $T^{\nu}$ are isomorphic to a fixed tree $T'$.
Let $\alpha'\in T'$ be a vertex and, using the notation from Def. \ref{defStableSupercurve},
we denote the set of special (i.e. marked and nodal) points of the stable supercurve with index $\nu$
on the $\alpha'$-sphere by $Y_{\alpha'}^{\nu}$. By Thm. \ref{thmSuperGromovCompactness},
the sequence $\left((\varphi_{\alpha'}^{\nu},\psi_{\alpha'}^{\nu}),Y_{\alpha'}^{\nu}\right)$
of holomorphic supercurves with marked points $Y_{\alpha'}^{\nu}$ has a subsequence which Gromov
converges to a stable supercurve $((\vvarphi_{\alpha'},\vpsi_{\alpha'}),\vz_{\alpha'})$.
We choose the subsequence such that we have this convergence for every $\alpha'\in T'$.
By the (Marked Points) axiom of Gromov convergence (Def. \ref{defSuperGromovConvergence}),
any marked point in $\vz_{\alpha'}$ on the $\alpha_i$-sphere of
$((\vvarphi_{\alpha'},\vpsi_{\alpha'}),\vz_{\alpha'})$ is then identified with either limit
\begin{align*}
z_i=\lim_{\nu\rightarrow\infty}(m_{\alpha_i}^{\nu})^{-1}(z_i^{\nu})\qquad\mathrm{or}\qquad
z_{\alpha'\beta'}=\lim_{\nu\rightarrow\infty}(m_{\alpha_i}^{\nu})^{-1}(z_{\alpha'\beta'}^{\nu})
\end{align*}
for a sequence of (original) marked points $z_i^{\nu}\in Y_{\alpha'}^{\nu}$ or nodal points
$z_{\alpha'\beta'}^{\nu}\in Y_{\alpha'}^{\nu}$ corresponding to the edge
$\alpha'E\beta'$ in $T'$, respectively.

For $\alpha'E\beta'$ in $T'$, we connect the limiting curves
$((\vvarphi_{\alpha'},\vpsi_{\alpha'}),\vz_{\alpha'})$ and
$((\vvarphi_{\beta'},\vpsi_{\beta'}),\vz_{\beta'})$ by joining the underlying trees
with the edge $\alpha'E\beta'$ and making the marked points $z_{\alpha'\beta'}$
and $z_{\beta'\alpha'}$ again nodal points. After connecting limiting curves for all $\alpha'\in T'$
this way, we finally obtain a single stable supercurve $((\vvarphi,\vpsi),\vz)$ to which,
by construction, (the considered subsequence of) $((\vvarphi^{\nu},\vpsi^{\nu}),\vz^{\nu})$ Gromov converges
in the sense of Def. \ref{defSuperGromovConvergenceStable}. This proves the theorem.
\end{proof}

Our next aim is to show uniqueness, up to equivalence, of the limit of a Gromov converging
sequence of stable holomorphic supercurves.
By Prp. \ref{prpInducedGromovConvergence}, to be established beforehand,
this follows from the uniqueness Thm. \ref{thmSuperGromovCompactnessUniqueness} for sequences of
holomorphic supercurves. The proof of Prp. \ref{prpInducedGromovConvergence}, in turn,
requires the following lemma, which we shall also need in Sec. \ref{subsecGromovTopology} below.
Motivated by Rem. \ref{remSuperGromovCompactnessStable}, we state it only in the case when
the tree homomorphisms $f^{\nu}:T\rightarrow T^{\nu}$ are fixed and equal to $f:T\rightarrow T'$.

\begin{Lem}[\cite{MS04}, Lem. 5.5.6]
\label{lemInducedGromovConvergence}
Let $T=(T,E,\Lambda)$ and $T'=(T',E',\Lambda')$ be $n$-labelled trees and $f:T\rightarrow T'$
be a surjective tree homomorphism. Let $\{m_{\alpha}^{\nu}\}_{\alpha\in T}^{\nu\in\bN}$ be a
collection of Möbius transformations that satisfy the (Rescaling), (Nodal Points), and
(Marked Points) axioms in Def. \ref{defSuperGromovConvergenceStable}. Then the following holds.
\begin{enumerate}
\renewcommand{\labelenumi}{(\roman{enumi})}
\item If $\alpha,\beta\in T$ are such that $\alpha\neq\beta$ and $f(\alpha)=f(\beta)$ then the
sequence $m_{\alpha\beta}^{\nu}:=(m_{\alpha}^{\nu})^{-1}\circ m_{\beta}^{\nu}$ converges to
$z_{\alpha\beta}$ u.c.s on $S^2\setminus\{z_{\beta\alpha}\}$.
\item If $\alpha,\beta\in T$ are such that $f(\alpha)\neq f(\beta)$ then
$z_{\alpha\beta}=\lim_{\nu\rightarrow\infty}(m_{\alpha}^{\nu})^{-1}(z^{\nu}_{f(\alpha)f(\beta)})$.
\item For each $\alpha\in T$ and each $i\in\{1,\ldots,n\}$,
$z_{\alpha i}=\lim_{\nu\rightarrow\infty}(m_{\alpha}^{\nu})^{-1}(z^{\nu}_{f(\alpha)i})$.
\end{enumerate}
\end{Lem}

Let $((\vvarphi,\vpsi),\vz)$ be a stable supercurve modelled over the labelled tree $(T,E,\Lambda)$
and let $T_0\subseteq T$ be a subtree. The \emph{restriction of $((\vvarphi,\vpsi),\vz)$ to $T_0$}
is then the stable supercurve with holomorphic supercurves $(\varphi_{\alpha},\psi_{\alpha})$
for $\alpha\in T_0$, nodal points $z_{\alpha\beta}$ for $\alpha,\beta\in T_0$ such that $\alpha E\beta$,
and marked points being the original marked points $(\alpha_i,z_i)$ with $\alpha_i\in T_0$ plus
the original nodal points $(\alpha,z_{\alpha\beta})$ for $\alpha\in T_0$ and $\beta\in T\setminus T_0$
such that $\alpha E\beta$.
For example, the stable supercurve $\left((\varphi_{\alpha'}^{\nu},\psi_{\alpha'}^{\nu}),Y_{\alpha'}^{\nu}\right)$
occurring in the proof of Thm. \ref{thmSuperGromovCompactnessStable} is the restriction of
$((\vvarphi^{\nu},\vpsi^{\nu}),\vz^{\nu})$ to $\{\alpha'\}\subseteq T'$. We have shown
how convergence of the individual restrictions add up to convergence of the original sequence.
The following proposition states the converse, in a sense, which is more involved: Gromov convergence
implies convergence of the (smallest) restrictions. It is needed in the proof that limits are
unique (up to equivalence).

\begin{Prp}
\label{prpInducedGromovConvergence}
Let $(A^{\nu},J^{\nu})\in\mA\times\mJ$ be a sequence that converges to $(A,J)$ in the
$C^{\infty}$-topology.
Let $((\vvarphi^{\nu},\vpsi^{\nu}),\vz^{\nu})\in\hat{\mM}_{0,d,n}(X;A^{\nu},J^{\nu})$
be a sequence of stable supercurves, modelled
over the (constant) labelled tree $(T',E',\Lambda')$, which Gromov converges to the stable
supercurve $((\vvarphi,\vpsi),\vz)\in\hat{\mM}_{0,d,n}(X;A,J)$,
modelled over $(T,E,\Lambda)$, via the surjective tree homomorphism
$f:T\rightarrow T'$. Then, in the cases $d=-2$ and $d=-1$, the sequence
$((\varphi^{\nu}_{\alpha'},\psi^{\nu}_{\alpha'}),Y^{\nu}_{\alpha'})$
of marked $(A^{\nu},J^{\nu})$-holomorphic supercurves Gromov converges in the sense of
Def. \ref{defSuperGromovConvergence} to the restriction of the stable supercurve
$((\vvarphi,\vpsi),\vz)$ to the subtree $f^{-1}(\alpha')\subseteq T$.
\end{Prp}

\begin{proof}
Let $T_0:=f^{-1}(\alpha')$ and denote by $((\vvarphi_0,\vpsi_0),\vz_0)$ the restriction of $((\vvarphi,\vpsi),\vz)$
to $T_0$. We prove that the sequences $((\varphi^{\nu}_{\alpha'},\psi^{\nu}_{\alpha'}),Y^{\nu}_{\alpha'})$
and $\{m^{\nu}_{\alpha}\}_{\alpha\in T_0}^{\nu}$ satisfy all axioms of Def. \ref{defSuperGromovConvergence}.
The (Rescaling) and (Marked Points) axioms follow directly from the (Rescaling), (Nodal Points) and (Marked Points)
axioms of Def. \ref{defSuperGromovConvergenceStable}.

To establish the (Map) and (Energy) axioms, we closely follow the argument of the proof of Prp. 5.5.2 in \cite{MS04}. Consider the set
$Z_{0\alpha}:=\{z_{\alpha\beta}\setsep\beta\in T_0,\,\alpha E\beta\}$.
Using Lem. \ref{lemInducedGromovConvergence}, we find
\begin{subequations}\begin{align}
\label{eqnLimMAB}
m_{\alpha\beta}(\vvarphi)=\lim_{\nu\rightarrow\infty}m_{\alpha'\beta'}(\vvarphi^{\nu})\;&,\qquad
\lim_{\varepsilon\rightarrow 0}\lim_{\nu\rightarrow\infty}
E(\varphi_{\alpha'}^{\nu}\circ m_{\alpha}^{\nu},B_{\varepsilon}(z_{\alpha\beta}))=0\\
m_{\alpha\beta}(\vpsi)=\lim_{\nu\rightarrow\infty}m_{\alpha'\beta'}(\vpsi^{\nu})\;&,\qquad
\lim_{\varepsilon\rightarrow 0}\lim_{\nu\rightarrow\infty}
E(\psi_{\alpha'}^{\nu}\circ m_{\alpha}^{\nu},B_{\varepsilon}(z_{\alpha\beta}))=0
\end{align}\end{subequations}
This shows that the sequence
$(\varphi_{\alpha'}^{\nu}\circ m_{\alpha}^{\nu},\psi_{\alpha'}^{\nu}\circ m_{\alpha}^{\nu})$
exhibits no bubbling near the point $z_{\alpha\beta}$ for every $\beta\in T\setminus T_0$ such that
$\alpha E\beta$ and, hence, converges to $(\varphi_{\alpha},\psi_{\alpha})$ u.c.s. on $S^2\setminus Z_{0\alpha}$
for every $\alpha\in T_0$, thus establishing the (Map) axiom.\\
To make the argument precise, ''no bubbling'' means that the $W^{1,\infty}$-norm of
$\varphi_{\alpha'}^{\nu}\circ m_{\alpha}^{\nu}$
stays bounded in a neighbourhood of $z_{\alpha\beta}$ and, by Prp. \ref{prpPsiLInfinityBounded}, the same is true
for the $L^{\infty}$-norm of $\psi_{\alpha'}^{\nu}\circ m_{\alpha}^{\nu}$.
We know, by the (Map) axiom in Def. \ref{defSuperGromovConvergenceStable}, that the sequence
converges to $(\varphi_{\alpha},\psi_{\alpha})$ u.c.s. on $S^2\setminus Z_{\alpha}$.
Assume the analogous statement for $S^2\setminus Z_{0\alpha}$ is false. Then the compactness
Prp. \ref{prpPhiPsiConverging} yields a contradiction, exactly as in the proof of
Lem. \ref{lemUniformImpliesCInfty}.
Finally, we note that the (Energy) axiom is verified as in the classical case.
\end{proof}

\begin{Thm}[Uniqueness of the Limit]
\label{thmSuperGromovCompactnessUniquenessStable}
Let $(A^{\nu},J^{\nu})\in\mA\times\mJ$ be a sequence that converges to $(A,J)$ in the
$C^{\infty}$-topology. Let
$((\vvarphi^{\nu},\vpsi^{\nu}),\vz^{\nu})\in\hat{\mM}_{0,d,n}(X;A^{\nu},J^{\nu})$ be a
sequence of stable supercurves that converges to two stable supercurves
$((\vvarphi,\vpsi),\vz)$ and $((\tilde{\vvarphi},\tilde{\vpsi}),\tilde{\vz})$.
Then, in the cases $d=-2$ and $d=-1$, $((\vvarphi,\vpsi),\vz)$ and
$((\tilde{\vvarphi},\tilde{\vpsi}),\tilde{\vz})$ are equivalent.
\end{Thm}

\begin{proof}
Being based on Prp. \ref{prpInducedGromovConvergence} and Thm. \ref{thmSuperGromovCompactnessUniqueness}, this is a straightforward
generalisation of the proof of the classical uniqueness theorem concerning
the underlying sequence $(\vvarphi^{\nu},\vz^{\nu})$ of stable maps
(Thm. 5.5.3 in \cite{MS04}).
\end{proof}

The following result is the analogon of Prp. \ref{prpSuperGromovConvergence} for sequences of
stable holomorphic supercurves. It is proved as Thm. 5.5.4(ii) in \cite{MS04}.

\begin{Prp}
\label{prpSuperGromovConvergenceStable}
Let $(A^{\nu},J^{\nu})\in\mA\times\mJ$ be a sequence that converges to $(A,J)$ in the
$C^{\infty}$-topology. Let
$((\vvarphi^{\nu},\vpsi^{\nu}),\vz^{\nu})\in\hat{M}_{0,d,n}(X;A^{\nu},J^{\nu})$
be a sequence of stable supercurves that Gromov converges to a stable supercurve
$((\vvarphi,\vpsi),\vz)\in\hat{M}_{0,d,n}(X;A,J)$.
Then, for sufficiently large $\nu$, the connected sum $\#_{\alpha\in T^{\nu}}\varphi_{\alpha}^{\nu}$
is homotopic to the connected sum $\#_{\alpha\in T}\varphi_{\alpha}$.
\end{Prp}

Consider a sequence of stable supercurves which Gromov converges to a stable supercurve
in $\hat{\mM}_{0,d,n}(X,\beta,E;A,J)$. By Prp. \ref{prpSuperGromovConvergenceStable},
we may assume without loss of generality that the sequence itself already lies in
$\hat{\mM}_{0,d,n}(X,\beta,E;A,J)$.

\subsection{The Gromov Topology}
\label{subsecGromovTopology}

Finally, we shall introduce a distance function for stable supercurves
as a generalisation of the one for stable maps, and a topology on the moduli spaces
$\overline{\mM}_{0,d,n}(X,\beta,E;A,J)$ for $d=-2$ and $d=-1$.
By means of that ''distance'' which, despite its name, is not symmetric and does
not satisfy the triangle inequality, we show that convergence with respect to the \emph{Gromov topology}
thus defined is equivalent to Gromov convergence.
Moreover, we prove that the moduli spaces are compact and metrisable in the case $d=-1$.
Let
\begin{align*}
\vx=(\vvarphi,\vpsi,\vz),\quad\vx'=(\vvarphi',\vpsi',\vz')\in\hat{\mM}_{0,d,n}(X,\beta,E;A,J)
\end{align*}
be two stable supercurves. Fixing a sufficiently small constant $\varepsilon>0$, we define the
\emph{distance between $\vx$ and $\vx'$} to be the real number
\begin{align}
\rho_{\varepsilon}(\vx,\vx')
:=\inf_{f:T\rightarrow T'}\inf_{\{m_{\alpha}\}}\rho_{\varepsilon}(\vx,\vx';f,\{m_{\alpha}\})
\end{align}
where the infimum is taken over all tuples $\{m_{\alpha}\}_{\alpha\in T}$ and all
surjective tree homomorphisms $f:T\rightarrow T'$ such that $f(\alpha_i)=\alpha'_i$
for all $i\in\{1,\ldots,n\}$ (mapping labels of $\vx$ to labels of $\vx'$).
If no such homomorphism exists, we set $\rho_{\varepsilon}(\vx,\vx'):=\infty$.
We define
\begin{align*}
\rho_{\varepsilon}(\vx,\vx';f,\{m_{\alpha}\})
&:=\sup_{\alpha E\beta}\abs{E_{\alpha}(\vvarphi,B_{\varepsilon}(z_{\alpha\beta}))
-E_{f(\alpha)}(\vvarphi',m_{\alpha}(B_{\varepsilon}(z_{\alpha\beta})))}\\
&\qquad\qquad+\sup_{\alpha E\beta}\abs{E_{\alpha}(\vpsi,B_{\varepsilon}(z_{\alpha\beta}))
-E_{f(\alpha)}(\vpsi',m_{\alpha}(B_{\varepsilon}(z_{\alpha\beta})))}\\
&\qquad\qquad+\sup_{\alpha\in T}\sup_{S^2\setminus B_{\varepsilon}(Z_{\alpha})}
\scal[d]{\varphi'_{f(\alpha)}\circ m_{\alpha}}{\varphi_{\alpha}}\\
&\qquad\qquad+\sup_{\alpha\in T}\sup_{S^2\setminus B_{\varepsilon}(Z_{\alpha})}
\scal[d]{\psi'_{f(\alpha)}\circ m_{\alpha}}{\psi_{\alpha}}\\
&\qquad\qquad+\sup_{\substack{\alpha\neq\beta\\f(\alpha)=f(\beta)}}
\sup_{S^2\setminus B_{\varepsilon}(z_{\alpha\beta})}
\scal[d]{m_{\beta}^{-1}\circ m_{\alpha}}{z_{\beta\alpha}}\\
&\qquad\qquad+\sup_{f(\alpha)\neq f(\beta)}
\scal[d]{m_{\beta}^{-1}(z'_{f(\beta)f(\alpha)})}{z_{\beta\alpha}}\\
&\qquad\qquad+\sup_{\substack{\alpha\in T\\1\leq i\leq n}}
\scal[d]{m_{\alpha}^{-1}(z'_{f(\alpha)i})}{z_{\alpha i}}
\end{align*}
Here, we may identify $\psi'_{f(\alpha)}\circ m_{\alpha}$ with a section of
$L_d\otimes_J\varphi_{\alpha}^*TX$ via the trivialisation from Lem. 4.9 in \cite{Gro11b},
provided that $\scal[d]{\varphi'_{f(\alpha)}\circ m_{\alpha}}{\varphi_{\alpha}}$
is sufficiently small, and then $\scal[d]{\psi'_{f(\alpha)}\circ m_{\alpha}}{\psi_{\alpha}}$
may be defined by any bundle metric, e.g. the one from (\ref{eqnH2}).
By the properties of holomorphic supercurves
such as conformal invariance (\ref{eqnSuperEnergyInvariance}) of the energy,
we immediately yield the following lemma.

\begin{Lem}
\label{lemDistanceEquivalence}
The distance functions $\vx'\mapsto\rho_{\varepsilon}(\vx,\vx')$ descend to the moduli space
$\overline{\mM}_{0,d,n}(X,\beta,E;A,J)$ of equivalence classes of stable supercurves:
\begin{align*}
\vx'\equiv\vy'\implies\rho_{\varepsilon}(\vx,\vx')=\rho_{\varepsilon}(\vx,\vy')
\qquad\mathrm{and}\qquad
\vx\equiv\vx'\implies\rho_{\varepsilon}(\vx,\vx')=0
\end{align*}
holds, where ''$\equiv$'' denotes the equivalence relation stated in Def. \ref{defStableSupercurveEquivalence}.
\end{Lem}

The functions $\rho_{\varepsilon}(\vx,\vx')$ characterise Gromov convergence in the following sense.

\begin{Lem}
\label{lemDistanceGromovConvergence}
Let $\vx=(\vvarphi,\vpsi,\vz)\in\hat{\mM}_{0,d,n}(X,\beta,E;A,J)$ be a stable supercurve.
Then, in the cases $d=-2$ and $d=-1$, there exists a constant $\varepsilon_0>0$ such that the
following holds for $0<\varepsilon<\varepsilon_0$. A sequence
$\vx^{\nu}=(\vvarphi^{\nu},\vpsi^{\nu},\vz^{\nu})\in\hat{\mM}_{0,d,n}(X,\beta,E;A,J)$
Gromov converges to $\vx$ if and only if the sequence of real numbers
$\rho_{\varepsilon}(\vx,\vx^{\nu})$ converges to zero.
\end{Lem}

\begin{proof}
Choose $\varepsilon>0$ such that $E(\varphi_{\alpha},B_{\varepsilon}(Z_{\alpha}))<\hbar/2$
and $B_{\varepsilon}(z_{\alpha\beta})\cap B_{\varepsilon}(z_{\alpha\gamma})=\emptyset$
for all $\alpha\in T$ and $\beta\neq\gamma$, where $\hbar>0$ denotes the minimal classical
energy for nonconstant $J$-holomorphic spheres. Then
\begin{align*}
E_{\alpha}(\vvarphi,B_{\varepsilon}(z_{\alpha\beta}))
&=E(\varphi_{\alpha},B_{\varepsilon}(z_{\alpha\beta}))+m_{\alpha\beta}(\vvarphi)\\
E_{\alpha}(\vpsi,B_{\varepsilon}(z_{\alpha\beta}))
&=E(\psi_{\alpha},B_{\varepsilon}(z_{\alpha\beta}))+m_{\alpha\beta}(\vpsi)
\end{align*}
holds for all $\alpha,\beta\in T$ such that $\alpha E\beta$.

If $\vx^{\nu}$ Gromov converges to $\vx$, then there are a surjective tree homomorphism
$f^{\nu}:T\rightarrow T^{\nu}$ and Möbius transformations $\{m_{\alpha}^{\nu}\}$
such that $\varphi^{\nu}_{f^{\nu}(\alpha)}\circ m_{\alpha}^{\nu}\rightarrow\varphi_{\alpha}$ u.c.s.
on $S^2\setminus Z_{\alpha}$ for all $\alpha\in T$. Hence
$d(\varphi^{\nu}_{f^{\nu}(\alpha)}\circ m_{\alpha}^{\nu},\varphi_{\alpha})\rightarrow 0$
for all $S^2\setminus B_{\varepsilon}(Z_{\alpha})$. Therefore, the corresponding term
in $\rho_{\varepsilon}(\vx,\vx^{\nu})$ converges to zero, and analogous for the
distance term involving $\psi^{\nu}$. By this and the (Energy) axiom in
Def. \ref{defSuperGromovConvergenceStable},
it is obvious that the energy terms in $\rho_{\varepsilon}(\vx,\vx^{\nu})$ also
converge to zero. For the other terms, this follows from Lem. \ref{lemInducedGromovConvergence}
and thus $\rho_{\varepsilon}(\vx,\vx^{\nu})\stackrel{\nu}{\rightarrow}0$
vanishes in the limit.

Conversely, suppose that $\rho_{\varepsilon}(\vx,\vx^{\nu})\stackrel{\nu}{\rightarrow}0$.
Then, for sufficiently large $\nu$, there exist a surjective tree homomorphism
$f^{\nu}:T\rightarrow T^{\nu}$ and Möbius transformations $\{m_{\alpha}^{\nu}\}_{\alpha\in T}$
such that $f^{\nu}(\alpha_i)=\alpha_i^{\nu}$ and
\begin{align*}
\rho^{\nu}:=\rho_{\varepsilon}(\vx,\vx^{\nu},f^{\nu},\{m_{\alpha}^{\nu}\})\leq
\rho_{\varepsilon}(\vx,\vx^{\nu})+2^{-\nu}
\end{align*}
We prove that this sequence satisfies all axioms from Def. \ref{defSuperGromovConvergenceStable}.
For (Marked Points), (Nodal Points) and (Rescaling) as well as the (Map) axiom
for $\vvarphi^{\nu}$, we refer to the proof of Lem. 5.5.8 in \cite{MS04}.

We show the (Map) axiom for $\vpsi^{\nu}$. Since $\rho^{\nu}\rightarrow 0$,
it follows that $\psi_{\alpha}^{\nu}:=\psi_{f^{\nu}(\alpha)}^{\nu}\circ m_{\alpha}^{\nu}$
converges to $\psi_{\alpha}$ uniformly on $S^2\setminus\bigcup_{\alpha E\beta}B_{\varepsilon}(z_{\alpha\beta})$.
By Lem. \ref{lemUniformImpliesCInfty}, this convergence is uniform with all derivatives.
From $\rho^{\nu}\rightarrow 0$, it is moreover clear that
$E(\psi^{\nu}_{\alpha},B_{2\varepsilon}(z_{\alpha\beta}))$ is uniformly bounded.
This, and the (Map) axiom for $\varphi_{\alpha}^{\nu}$ already established,
shows that the hypotheses of Prp. \ref{prpPsiLInfinityBounded} are satisfied on
$B_{2\varepsilon}(z_{\alpha\beta}))\setminus B_{\delta}(z_{\alpha\beta})$
for $0<\delta<\varepsilon$, providing a convergent subsequence on any compact subset.
On $B_{2\varepsilon}(z_{\alpha\beta}))\setminus B_{\varepsilon}(z_{\alpha\beta})$,
this limit agrees with $\psi_{\alpha}$, and hence, by unique continuation
(Lem. 3.5 in \cite{Gro11b}), the limit agrees with $\psi_{\alpha}$ whereever
both are defined, and thus the sequence $\psi_{\alpha}^{\nu}$ itself converges to
$\psi_{\alpha}$ u.c.s. on $B_{\varepsilon}(z_{\alpha\beta})\setminus\{z_{\alpha\beta}\}$.
This proves the (Map) axiom for $\vpsi^{\nu}$.

Finally note that, again by $\rho^{\nu}\rightarrow 0$ and
by the choice of $\varepsilon$, we have
\begin{align*}
m_{\alpha\beta}(\vvarphi)+E(\varphi_{\alpha}, B_{\varepsilon}(z_{\alpha\beta}))
&=\lim_{\nu\rightarrow\infty}E_{f^{\nu}(\alpha)}
(\vvarphi^{\nu},m_{\alpha}^{\nu}(B_{\varepsilon}(z_{\alpha\beta})))\\
m_{\alpha\beta}(\vpsi)+E(\psi_{\alpha}, B_{\varepsilon}(z_{\alpha\beta}))
&=\lim_{\nu\rightarrow\infty}E_{f^{\nu}(\alpha)}
(\vpsi^{\nu},m_{\alpha}^{\nu}(B_{\varepsilon}(z_{\alpha\beta})))
\end{align*}
Taking the limit $\varepsilon\rightarrow 0$ on both sides, we yield the
(Energy) axiom.

Altogether, we have shown that the sequence $\vx^{\nu}=(\vvarphi^{\nu},\vpsi^{\nu},\vz^{\nu})$
Gromov converges to $\vx=(\vvarphi,\vpsi,\vz)$ provided that
$\rho_{\varepsilon}(\vx,\vx^{\nu})\stackrel{\nu}{\rightarrow}0$, which completes the proof.
\end{proof}

The functions $\rho_{\varepsilon}(\vx,\vx')$ satisfy the following substitute for
the triangle inequality.

\begin{Lem}
\label{lemDistanceTriangle}
For $d=-2$ or $d=-1$, let $\vx\in\hat{\mM}_{0,d,n}(X,\beta,E;A,J)$ be a stable supercurve
and $\varepsilon_0>0$ be as in Lem. \ref{lemDistanceGromovConvergence}.
Then the following conclusion holds for every stable supercurve $\vx'\in\hat{\mM}_{0,d,n}(X,\beta,E;A,J)$,
every sequence $\vx^{\nu}\in\hat{\mM}_{0,d,n}(X,\beta,E;A,J)$ of stable supercurves and
$0<\varepsilon<\varepsilon_0$.
If $\vx'$ satisfies $\rho_{\varepsilon}(\vx,\vx')<\varepsilon$
and $\vx^{\nu}$ Gromov converges to $\vx'$ then
\begin{align*}
\limsup_{\nu\rightarrow\infty}\rho_{\varepsilon}(\vx,\vx^{\nu})\leq\rho_{\varepsilon}(\vx,\vx')
\end{align*}
\end{Lem}

\begin{proof}
Compared to the proof of the underlying classical case concerning stable maps
(Lem. 5.5.9 in \cite{MS04}), no new ideas are required here,
and we omit the details.
\end{proof}

We will now define the Gromov topology on the moduli space $\overline{\mM}_{0,d,n}(X,\beta,E;A,J)$
for $d=-2$ and $d=-1$ and then show, by the properties of the distance function
just established, that this topology is second countable and Hausdorff, and that convergence
is equivalent to Gromov convergence.
Moreover, we shall prove that the moduli spaces are compact and metrisable in the case $d=-1$.

In general, let $(M,\mU)$ be a Hausdorff topological space, which is first countable.
Then limits are unique and the closure $cl(A)$ of a subset $A\subseteq M$ is
exactly the set of limit points of convergent sequences in $A$.
We define $\mC(\mU)\subseteq M\times M^{\bN}$ to be the set of all pairs
$(x_0,(x_n)_n)$ of elements $x_0\in M$ and sequences $x_n\in M$ such that
$x_n$ converges to $x_0$.
Conversely, let $\mC\subseteq M\times M^{\bN}$ be an arbitrary collection of sequences.
We define $\mU(\mC)\subseteq 2^M$ to be the set of all subsets $U\subseteq M$ that satisfy
\begin{align}
\label{eqnUC}
(x_0,(x_n)_n)\in\mC\cap(U\times M^{\bN})\implies
\exists N\in\bN\;\forall n\geq N: x_n\in U
\end{align}
This can be seen to be indeed a topology.

\begin{Lem}[\cite{MS04}, Lem. 5.6.5]
\label{lemSequenceTopology}
Let $M$ be a set and $\mC\subseteq M\times M^{\bN}$ be a collection of sequences in $M$
that satisfies the uniqueness axiom:
\begin{align*}
(x_0,(x_n)_n)\in\mC\quad\mathrm{and}\quad(y_0,(x_n)_n)\in\mC\quad\implies\quad x_0=y_0
\end{align*}
Suppose that for every $x\in M$ there exists a
constant $\varepsilon_0(x)>0$ and a collection of functions
$M\rightarrow[0,\infty]:x'\mapsto\rho_{\varepsilon}(x,x')$ for $0<\varepsilon<\varepsilon_0(x)$
satisfying the following conditions.
\begin{enumerate}
\renewcommand{\labelenumi}{(\roman{enumi})}
\item If $x\in M$ and $0<\varepsilon<\varepsilon_0(x)$ then $\rho_{\varepsilon}(x,x)=0$.
\item If $x\in M$, $0<\varepsilon<\varepsilon_0(x)$, and $(x_n)_n\in M^{\bN}$ then
\begin{align*}
(x,(x_n)_n)\in\mC\iff\lim_{n\rightarrow\infty}\rho_{\varepsilon}(x,x_n)=0
\end{align*}
\item If $x\in M$, $0<\varepsilon<\varepsilon_0(x)$, and $(x',(x_n)_n)\in\mC$, then
\begin{align*}
\rho_{\varepsilon}(x,x')<\varepsilon\implies
\limsup_{n\rightarrow\infty}\rho_{\varepsilon}(x,x_n)\leq\rho_{\varepsilon}(x,x')
\end{align*}
\end{enumerate}
Then $\mC=\mC(\mU(\mC))$. Moreover, the topology $\mU(\mC)$ is first countable and Hausdorff.
\end{Lem}

By Theorem \ref{thmSuperGromovCompactnessUniquenessStable} on unique limits,
Gromov convergence for sequences of stable supercurves
as defined in Def. \ref{defSuperGromovConvergenceStable}
descends to the moduli space $M:=\overline{\mM}_{0,d,n}(X,\beta,E;A,J)$.
In other words, it makes sense to consider Gromov convergent sequences of equivalence
classes. We define the \emph{Gromov topology} on $M$ to be the topology $\mU(\mC)$
as in (\ref{eqnUC}), where we take $\mC$ to be the set of all Gromov convergent
sequences (to be more precise, the set of all pairs $([\vx],[\vx^{\nu}])\in M\times M^{\bN}$
such that $[\vx^{\nu}]$ Gromov converges to $[\vx]$).

We conclude this chapter with the following main theorem on the Gromov topology
as advertised above, whose proof is based on Lem. \ref{lemSequenceTopology}.
Consult \cite{Mun00} for the necessary background and results on topological spaces,
such as the countability axioms and Urysohn's metrisation theorem.

\begin{Thm}
Let $d=-2$ or $d=-1$. Then the Gromov topology on the moduli space
$\overline{\mM}_{0,d,n}(X,\beta,E;A,J)$ satisfies the following properties.
\begin{enumerate}
\renewcommand{\labelenumi}{(\roman{enumi})}
\item It is first countable and Hausdorff, and
a sequence in $\overline{\mM}_{0,d,n}(X,\beta,E;A,J)$ converges with respect to the Gromov topology
if and only if it Gromov converges.

\item The Gromov topology even satisfies the second countability axiom.

\item Moreover, in the case $d=-1$, the moduli space $\overline{\mM}_{0,d,n}(X,\beta,E;A,J)$
is compact and metrisable.
\end{enumerate}
\end{Thm}

\begin{proof}
By Thm. \ref{thmSuperGromovCompactnessUniquenessStable}, Gromov convergent sequences have
unique limits. By this result and, moreover, Lemmas \ref{lemDistanceEquivalence},
\ref{lemDistanceGromovConvergence} and \ref{lemDistanceTriangle}, the hypotheses of
Lem. \ref{lemSequenceTopology} are satisfied such that (i) follows.

We prove that $\overline{\mM}_{0,d,n}(X,\beta,E;A,J)$ is second countable.
Fix an $n$-labelled tree $(T,E,\Lambda)$. Corresponding to the decomposition
of the homology class $\beta$ into sums $\beta=\sum_{\alpha\in T}\beta_{\alpha}$,
the subset of stable supercurves in $\hat{\mM}_{0,d,n}(X,\beta,E;A,J)$
modelled over $(T,E,\Lambda)$ is a countable union of compact subsets of a separable
Banach manifold, and thus has a countable dense subset. The same remains true upon
factoring out the equivalence relation from Def. \ref{defStableSupercurveEquivalence}.
Since there are only finitely many $n$-labelled trees over which there exist stable
supercurves representing $\beta$ (cf. Rem. \ref{remSuperGromovCompactnessStable}),
the whole moduli space $\hat{\mM}_{0,d,n}(X,\beta,E;A,J)$ contains a countable
dense subset.
The union of the countable neighbourhood bases of the elements of such a countable
dense subset is a countable basis for the Gromov topology, thus
establishing (ii).

In the case $d=-1$, Thm. \ref{thmSuperGromovCompactnessStable} asserts that every sequence in
$\overline{\mM}_{0,d,n}(X,\beta,E;A,J)$ has a Gromov convergent subsequence.
Hence, the moduli space is sequentially compact with respect to the Gromov
topology. Now every second countable sequentially compact topological space is compact and,
moreover, every compact Hausdorff space is normal. By Urysohn's metrisation theorem,
every normal space with a countable basis is metrisable. This concludes the proof of (iii).
\end{proof}

\bibliographystyle{plain}

\addcontentsline{toc}{section}{References}

\begin{thebibliography}{10}

\bibitem{Alt02}
H.~Alt.
\newblock {\em Lineare Funktionalanalysis}.
\newblock Springer, 2002.

\bibitem{BP92}
R.~Benedetti and C.~Petronio.
\newblock {\em Lectures on Hyperbolic Geometry}.
\newblock Springer, 1992.

\bibitem{CZ52}
A.~Calderon and A.~Zygmund.
\newblock On the existence of certain singular integrals.
\newblock {\em Acta Math.}, 88:85--139, 1952.

\bibitem{CZ56}
A.~Calderon and A.~Zygmund.
\newblock On singular integrals.
\newblock {\em Amer. J. Math.}, 78:289--309, 1956.

\bibitem{CJLW05}
Q.~Chen, J.~Jost, J.~Li, and G.~Wang.
\newblock Regularity theorems and energy identities for {D}irac-harmonic maps.
\newblock {\em Math. Z.}, 251:61--84, 2005.

\bibitem{CGMS02}
K.~Cieliebak, A.~Gaio, I.~Mundet, and D.~Salamon.
\newblock The symplectic vortex equations and invariants of {H}amiltonian group
  actions.
\newblock {\em J. Symplectic Geom.}, 1(3):543--646, 2002.

\bibitem{Dob06}
M.~Dobrowolski.
\newblock {\em {A}ngewandte {F}unktionalanalysis}.
\newblock Springer, 2006.

\bibitem{Dra04}
D.~Dragnev.
\newblock Fredholm theory and transversality for noncompact pseudoholomorphic
  maps in symplectizations.
\newblock {\em Communications on Pure and Applied Mathematics}, 57(6):726--763,
  2004.

\bibitem{EGH00}
Y.~Eliashberg, A.~Givental, and H.~Hofer.
\newblock {\em Introduction to symplectic field theory}, pages 560--673.
\newblock Birkhäuser, 2000.

\bibitem{Gro11a}
J.~Groeger.
\newblock Holomorphic supercurves and supersymmetric sigma models.
\newblock {\em Journal of Mathematical Physics}, 52(12), 2011.

\bibitem{Gro11b}
J.~Groeger.
\newblock Transversality for holomorphic supercurves.
\newblock Preprint, Humboldt-Universität zu Berlin, 2011.

\bibitem{Gro85}
M.~Gromov.
\newblock Pseudoholomorphic curves in symplectic manifolds.
\newblock {\em Invent. Math.}, 82:307--347, 1985.

\bibitem{Hat01}
A.~Hatcher.
\newblock {\em Algebraic Topology}.
\newblock Cambridge University Press, 2001.

\bibitem{MS04}
D.~McDuff and D.~Salamon.
\newblock {\em $J$-holomorphic Curves and Symplectic Topology}.
\newblock American Mathematical Society, 2004.

\bibitem{Mun00}
J.~Munkres.
\newblock {\em Topology}.
\newblock Prentice Hall, 2000.

\bibitem{RS01}
J.~Robbin and D.~Salamon.
\newblock Asymptotic behaviour of holomorphic strips.
\newblock {\em Ann. I.H. Poincare}, 18(5):573--612, 2001.

\bibitem{Wer02}
D.~Werner.
\newblock {\em Funktionalanalysis}.
\newblock Springer, 2002.

\end{thebibliography}

\end{document}